\documentclass[11pt,reqno]{amsart}
\usepackage{amsfonts}
\usepackage{amssymb,color}
\usepackage{amssymb}
\usepackage{setspace}
\usepackage[pagewise]{lineno}

\usepackage[colorlinks=true]{hyperref}
\hypersetup{urlcolor=blue, citecolor=blue}

\usepackage[margin=2.6cm, a4paper]{geometry}

\def\normo#1{\left\|#1\right\|}
\def\normb#1{\big\|#1\big\|}

\def\norm#1{\|#1\|}
\def\jb#1{\langle#1\rangle}

\def\wh#1{\widehat{#1}}

\newcommand{\R}{{\mathbb R}}

\newcommand{\Z}{{\mathbb Z}}
\newcommand{\ft}{{\mathcal{F}}}

\newcommand{\les}{{\lesssim}}
\newcommand{\ges}{{\gtrsim}}

\DeclareMathOperator{\supp}{supp}
\DeclareMathOperator{\diam}{diam}

\def\jb#1{\langle#1\rangle}
\def\norm#1{\|#1\|}
\def\normo#1{\left\|#1\right\|}
\def\normb#1{\big\|#1\big\|}

\def\wh#1{\widehat{#1}}

\newcommand{\e}{\varepsilon}

\newcommand{\I}{\infty}

\newcommand{\EQ}[1]{\begin{equation}\begin{split} #1 \end{split}\end{equation}}

\newcommand{\Del}[1]{}
\newcommand{\CAS}[1]{\begin{cases} #1 \end{cases}}

\numberwithin{equation}{section}

\newtheorem{theorem}{Theorem}[section]
  \newtheorem{proposition}[theorem]{Proposition}
  \newtheorem{lemma}[theorem]{Lemma}
  \newtheorem{corollary}[theorem]{Corollary}

\theoremstyle{remark}

\newtheorem{remark}[theorem]{Remark}

\theoremstyle{definition}
  \newtheorem{definition}[theorem]{Definition}

\begin{document}
\title[Scattering for mass-critical NLKG]{Scattering for the mass-critical nonlinear Klein-Gordon equations in three and higher dimensions}
\author[Cheng, Guo and Masaki]{Xing Cheng$^{*}$, Zihua Guo$^{**}$, and Satoshi Masaki$^{***}$}

\thanks{$^*$ College of Science, Hohai University, Nanjing 210098, Jiangsu, China. \texttt{chengx@hhu.edu.cn}}
\thanks{$^{**}$ School of Mathematical Sciences, Monash University, VIC 3800,  Australia. \texttt{Zihua.Guo@monash.edu}  }
\thanks{$^{***}$ Department of Systems Innovation, Graduate School of Engineering Science, Osaka University, Toyonaka Osaka, 560-8531, Japan. \texttt{masaki@sigmath.es.osaka-u.ac.jp}}
\thanks{$^{*}$    Xing Cheng has been partially supported by the NSF grant of China (No. 11526072).}	
\thanks{$^{**}$  Zihua Guo was supported by the ARC project (No. DP170101060).}
\thanks{$^{***}$ Satoshi Masaki was supported by JSPS KAKENHI Grant Numbers JP17K14219, JP17H02854, JP17H02851, and JP18KK0386.}

\begin{abstract}
In this paper we consider the mass-critical nonlinear Klein-Gordon equations in three and higher dimensions.
We prove the dichotomy between scattering and blow-up below the ground state energy in the focusing case,
and the energy scattering in the defocusing case.
We use the concentration-compactness/rigidity method developed by C. E. Kenig and F. Merle. The main novelty from the work of R. Killip, B. Stovall, and M. Visan [Trans. Amer. Math. Soc. {\bf364} (2012)]
 is to approximate the large scale (low-frequency) profile by the solution of the mass-critical nonlinear Schr\"odinger equation when the nonlinearity is not algebraic.
\bigskip

\bigskip


\end{abstract}
\maketitle


\section{Introduction}\label{se1}
In this paper, we consider the scattering problem for the mass-critical nonlinear Klein-Gordon equations (NLKG) on $\mathbb{R}^d$:
\begin{equation}\label{eq1.1}
\begin{cases}
- \partial_t^2 u + \Delta u - u = \mu |u|^\frac4d u ,\\
u(0,x) =u_0(x),\\
\partial_t u(0,x) =u_1(x) ,
\end{cases}
\end{equation}
where $u: \mathbb{R}\times \mathbb{R}^d \to \mathbb{R}$, $d\ge 3$, in both the defocusing case ($\mu=1$) and focusing case ($\mu=-1$). The NLKG equation
 is a fundamental model in mathematical physics and has been extensively studied in a large amount of literatures, for example, see \cite{NS2,Str,T2} and references therein. A major effort was recently devoted to the scattering problem.

An important class of nonlinearity is the power type nonlinearity $\mu |u|^{p-1} u$, where $p > 1$. 
Although the NLKG equation with the power type nonlinearity do not have a scaling structure, we can find that in the massless case, that is for
the corresponding wave equation, it has the scaling structure $u(t,x) \mapsto \lambda^{- \frac2{p-1}}u (\lambda t, \lambda x)$.
The scaling leaves the $\dot{H}^{s_c}_x$-norm invariant, where $s_c = \frac{d}2-\frac{2}{p-1}$. As blow-up is associated with the small spatial scale that is when $\lambda \to \infty$ and the mass term shrinks to 0 under this scaling. Therefore, it is natural to view $s_c$ as the critical regularity of the NLKG. In general, there are two critical indices for $p$: mass-critical index $p = 1 + \frac4d$ and energy-critical index $p = 1 + \frac4{d-2}$ when $d \ge 3$.
These two indices correspond to $s_c=0$ and $s_c=1$, respectively.
On the global dynamics there are many studies: for defocusing inter-critical cases $1+\frac4d<p<1+\frac4{d-2}$ (\cite{GV,GV1,N-1}), defocusing energy-critical cases (\cite{N-2}) and focusing inter-critical and energy-critical cases (\cite{IMN,IMN1,In,KNS,NS1,NS2,NS4}).
For mass critical cases, energy scattering was studied by R. Killip, B. Stovall, and M. Visan \cite{KSV1} for the two dimensional case and recently by M. Ikeda, T. Inui, and M. Okamoto \cite{IIO} for the one dimensional case.
The two works used the concentration-compactness/rigidity method developed by Kenig-Merle \cite{KM,KM1}.

The purpose of this paper is to study the mass-critical NLKG equations and prove energy scattering in three and higher dimensions. The mass-critical NLKG equation \eqref{eq1.1} has a conservation of energy
\begin{align*}
E\left(u,\partial_t u\right) :=
  \int_{\mathbb{R}^d} \frac12 |\partial_t u(t,x)|^2 + \frac12 |\nabla u(t,x)|^2 + \frac12 |u(t,x)|^2 + \frac{\mu d}{2(d+2)} |u(t,x)|^\frac{2(d+2)}d\, \mathrm{d}x,
\end{align*}
and also a conservation of momentum
\begin{align*}
P\left(u, \partial_t u \right) : =   \int_{\mathbb{R}^d} \partial_t u \cdot \nabla u \,\mathrm{d}x.
\end{align*}
Thus a natural phase space for NLKG is the energy space $H^1\times L^2$.

In the defocusing case, the conserved energy  immediately gives us the global existence of solutions.
On the other hand, in the focusing case, there is a global-existence/blowup dichotomy.
The ground state solution,  a static solution $u(t,x)=Q(x)$ to the NLKG equation plays a crucial role in the dichotomy.
Here, $Q(x)\in H^1$ is a positive radial solution to the nonlinear elliptic equation
\begin{align}\label{eq1.4}
\Delta Q - Q = - Q^{1 + \frac4d}.
\end{align}
Global well-posedness vs blow-up for the solutions under $E(u,u_t)< E(Q, 0 )$ was given essentially in \cite{PS}, where the threshold of $\norm{u_0}_{L^2}$
is used to discriminate the solutions.  More precisely, one has global well-posedness when $\norm{u_0}_{L^2} <  \norm{Q}_{L^2}$ and blowup when  $\norm{u_0}_{L^2}>\norm{Q}_{L^2}$.
However, in both defocusing and focusing cases, scattering needs more effort due to the mass-criticality.
The main result of this paper is to establish the scattering for the global solutions.
\begin{theorem}\label{th1.4}
Assume $(u_0,u_1) \in H_x^1(\R^d) \times L_x^2(\R^d)$, $d\geq 3$. We have

(i) if $\mu=1$, then the global solution to \eqref{eq1.1} scatters in energy space in both time directions,
that is, there exist $u_\pm \in C_t^0 H_x^1 \cap C^1_t L_x^2 $ which are the solutions of the linear Klein-Gordon equation such that
\begin{align*}
\left\| u(t) - u_\pm(t) \right\|_{H_x^1} +  \left\| \partial_t u(t) - \partial_t u_\pm(t) \right\|_{L_x^2} \to 0,\quad \text{ as } t \to \pm \infty.
\end{align*}

(ii) if $\mu=-1$, we assume further $E(u_0,u_1) < E(Q, 0)$, then the solution $u$ to \eqref{eq1.1} exists globally and scatters in the energy space when $\norm{u_0}_{L^2} <  \norm{Q}_{L^2}$; and it blows up in finite time  when $\norm{u_0}_{L^2}>\norm{Q}_{L^2}$. Also, the case when $\|u_0 \|_{L^2 } = \|Q\|_{L^2}$ is impossible. 
\end{theorem}

To prove the scattering part of the above theorem, we use the ``Kenig-Merle roadmap'' as in \cite{KM,KM1} and \cite{KSV1}.
Our main technical development lies in the linear and nonlinear profile decomposition in higher dimensions.
This is a key tool to prove  the existence of a non-scattering solution with the minimal energy, which is so-called a \emph{minimal energy blow-up solution}.

First, due to the mass-criticality, we need to establish the linear profile decomposition associated to the linear Klein-Gordon equation in higher dimensions at the $L^2$-critical level.  More precisely, we need to characterise the defect of the compactness of the Strichartz estimate
\EQ{
\normo{e^{it\jb{\nabla}}f}_{L_{t,x}^{2+\frac{4}{d}}(\R\times \R^d)}\les \norm{f}_{H^1}.
}
Since the right hand side can be replaced by the $H^{1/2}$-norm (see Remark \ref{re:H1/2}, below),
assuming bounded data in $H^1$, we can handle the high frequency easily due to the room of regularity. For the low frequency, it is much more complicated. This can be seen by the fact that the low frequency limit of Klein-Gordon equation is indeed the Schr\"odinger equation.
In fact, for any $\varphi \in H^1$, we have
\EQ{
e^{ it\lambda^2}e^{- it\lambda^2\jb{\lambda^{-1}\nabla}} \varphi  \to e^{i\frac{t}{2}\Delta} \varphi \text{ in $H^1$}, \text{ as } \lambda\to \infty,
}
by the asymptotic expansion
\begin{align}\label{eq1.4v138}
\lambda^2 \left( \langle \lambda^{-1} \xi \rangle -1 \right) = \tfrac12 |\xi|^2 + O\left( \lambda^{-2} |\xi|^4 \right), \text{ as } \lambda \to \infty.
\end{align}
Thus we have to take into account more symmetries for low frequency.
In this example, the Fourier transform concentrates at the origin.
The Lorentz boost is also involved when the Fourier transform concentrates to another point.
We will rely on some refined Strichartz estimates which are derived by the bilinear Strichartz estimates. We slightly simplify the argument in \cite{KSV1} (See Remark \ref{rem:improve}).

Second, we use the nonlinear profile decomposition to construct the minimal mass non-scattering solutions.  For this step, the mass-critical NLS
\EQ{\label{eq:mcNLS}
i\partial_t w+ \frac12 \Delta w  = \mu C_d |w|^{\frac{4}d}w
}
serves as approximate equation to that determines the long-time behaviour to the large scale nonlinear profile of the NLKG equation, where
\[
	C_d := \tfrac{\Gamma\left( \frac2d + \frac32 \right) }{  \sqrt{\pi} \Gamma\left( \frac2d + 2\right)}.
\]
The connection between \eqref{eq1.1} and \eqref{eq:mcNLS} in the scattering problem is previously studied.
K. Nakanishi \cite{N1} proved that the scattering of the NLKG implies the scattering of the corresponding NLS equation.  Conversely, R. Killip, B. Stovall, and M. Visan \cite{KSV1} used the scattering results of the mass-critical NLS equation to show the scattering of NLKG in two dimension.  This was extended to one dimensional case by M. Ikeda, T. Inui, and M. Okamoto in \cite{IIO}. Unlike the one- and two-dimensional cases, some new difficulty is caused by the fact that the power of the nonlinear term is of fractional order in the higher dimensions.  The limit NLS equation is not as obvious as  in one and two dimensional case.
The difficulty can be summarized  as specifying the constant $C_d$.
By the technique developed by the third author and his collaborators \cite{MMU,MS0,MS,MSU}, we can overcome this difficulty. In these works, they
introduce an expansion of homogeneous nonlinearity to pick up the resonant term from the non-algebraic nonlinear term.  By these ideas we derive the limit NLS equation and then use it to construct the minimal energy blow-up solutions.
Note that a similar technique was developed in \cite{MNO,MN,N,N1}.

Finally, in the rigidity part, we exclude the existence of the critical element by a virial type monotonicity argument.
This part is done by an argument in \cite{IMN,KSV1}. We give a proof of this part for self-containedness.

\section{Preliminary}\label{se2v90}

\subsection{Definition and notations}

We use $C$ to denote some universal constant which may change from line to line.
For $X, Y \in \mathbb{R}$, $X\lesssim Y$
means that there exists a constant $C$ such that $X\leq CY$, similarly for $X\ges Y$.  We use $\les_{A,\epsilon}$ and $\ges_{A,\epsilon}$  to indicate that the implicit constant depends on $A,\epsilon$.  $X\sim Y$ means $X\les Y$ and $X\ges Y$.  For $a\in \R$,
$a+$ (resp. $a-$) denotes  $a+\varepsilon$ (resp $a-\varepsilon$) for any sufficiently small $\varepsilon>0$, and $\langle a\rangle =(1+|a|^2)^{\frac12}$.

For a function $f\in L_{loc}^1(\R^d)$, we use $\wh f$ or $\ft(f)$ to denote the spatial Fourier transform of $f$:
$\mathcal{F}(f) (\xi) = \hat{f}(\xi) =(2 \pi)^{-\frac{d}2 } \int_{\R^d}e^{- i x \xi} f(x) \,\mathrm{d}x$.
Let $\varphi\in C_0^\infty(\R)$ be a real-valued, non-negative, even, and radially-decreasing function such that
\begin{align*}
\varphi(\xi) =
\begin{cases}
1, |\xi| \leq 1, \\
0, |\xi| \ge \frac54,
\end{cases}
\end{align*}
and define $\chi(\xi)=\varphi(\xi)-\varphi(2\xi)$. For a dyadic number $N\in 2^{\Z_+}$, we define the Littlewood-Paley projectors:
 $\widehat{P_1f}(\xi):=\varphi(\xi)\widehat{f}(\xi)$ and for $N\geq 2$,
\EQ{
\widehat{P_Nf}(\xi):=\chi\left(\frac{\xi}{N} \right)\widehat{f}(\xi), \quad \widehat{P^{\pm}_Nf}(\xi):=\chi\left(\frac{\xi}{N} \right)1_{\pm \xi\geq 0}\cdot \widehat{f}(\xi).
}
For $\Omega\subseteq \R^d$, we also define the Littlewood-Paley projector $P_\Omega =\ft^{-1}1_\Omega(\xi)\ft$.  We define the Fourier multiplier $m(\nabla)=\ft^{-1}m(\xi)\ft$. In particular, $\jb{\nabla}$ (resp. $D^s$) is the Fourier multiplier with symbol $\jb{\xi}=(1+|\xi|^2)^{\frac12}$ (resp. $|\xi|^s$).

We use $L^p$ to denote the Lebesgue space with a norm $\norm{\cdot}_{p}:=\norm{\cdot}_{L^p}$ and $L^p_tL_x^q$ to denote the mixed norm Lebesgue space with $\norm{f}_{L^p_tL_x^q}=\normb{\norm{f(t,\cdot)}_{L_x^q}}_{L_t^p}$. $H^s$ (and $\dot H^s$) denotes the standard (homogeneous) Sobolev space.

It is convenient for us to rewrite \eqref{eq1.1} into the first order. Let $v = u + i \langle \nabla \rangle^{-1} \partial_t u$, then the equation for $v$ is
\begin{align}\label{eq2.2}
\begin{cases}
i \partial_t v -  \langle \nabla \rangle v =  \mu \langle \nabla \rangle^{-1}\left(  \left|\Re v\right|^{\frac{4}d } \Re v \right),\\
v(0,x) = v_0(x) \in H^1(\mathbb{R}^d),
\end{cases}
\end{align}
We will use these two equivalent forms interchangeably.   We use $u$ to denote solution of \eqref{eq1.1} and $v$ the corresponding solution of \eqref{eq2.2}, the scattering norms and energies are defined to be
\begin{align*}
S_I(u) & = S_I(v) = \int_{I}\int_{\mathbb{R}^d} \left|\Re v(t,x)\right|^{\frac{2(d+2)}d}\,\mathrm{d}x \mathrm{d}t,\\
E(u(t)) & = E(v(t)) = \int_{\mathbb{R}^d} \frac12 \left|\langle \nabla_x  \rangle v(t,x)\right|^2 + \mu \frac{d }{2(d+2)} \left|\Re v(t,x)\right|^\frac{2(d+2)}d \,\mathrm{d}x.
\end{align*}

\subsection{Well-posedness theory}\label{se3}
In this subsection we collect some Strichartz estimates and well-posedness results that will be used in this paper.
First we recall the dispersive estimate for the Klein-Gordon propagator (see e.g. \cite{Br,GV,GV1,MSW}).
\begin{lemma} For any dyadic number $N\geq 1$, we have
\begin{align*}
\normo{e^{it\jb{\nabla}}P_N f}_{L_x^\infty(\R^d)}\les & |t|^{-\frac{d}{2}} N^{\frac{d+2}{2}}\norm{f}_{L_x^1},
\intertext{ and }
\normo{e^{it\jb{\nabla}}P_N f}_{L_x^\infty(\R^d)}\les & |t|^{-\frac{d-1}{2}} N^{\frac{d+1}{2}}\norm{f}_{L_x^1}.
\end{align*}
\end{lemma}
By the above dispersive estimate, we can get the Strichartz estimate.
The Strichartz estimate of the Klein-Gordon equation has been studied in many literatures, see \cite{Br,GV,GV1,IMN,MSW} and the references therein. In general, the Klein-Gordon propagator behaves like wave for high frequency and Sch\"odinger for low frequency.
\begin{definition}
We say that a pair $(q,r)$ is wave-admissible if
\EQ{
2 \le q, r \le \infty, \,  \frac1q \le \frac{d-1}2 \left(\frac12 - \frac1r \right), \, (q,r,d)\neq (2,\infty, 3);
}
and Schr\"odinger-admissible if
\EQ{
2 \le q, r \le \infty, \,  \frac1q \le \frac{d}2 \left(\frac12 - \frac1r \right), \, (q,r,d)\neq (2,\infty, 2).
}
If the equality holds, then we say $(q,r)$ is sharp wave (or Schr\"odinger)-admissible.
\end{definition}
\begin{lemma}[Homogeneous Strichartz estimate]
Assume $(q,r)$ is Schr\"odinger-admissible. We have
\EQ{
\normo{e^{it\jb{\nabla}}P_N f}_{L_t^qL_x^r(\R\times \R^d)}\les   & N^{\beta(q,r)}\norm{f}_{L_x^2}
}
where
\EQ{
\beta(q,r)=
\CAS{
\frac{d+2}{2} \left(\frac{1}{2}-\frac{1}{r} \right), \quad \mbox{if $(q,r)$ is sharp Schr\"odinger-admissible}, \\
\frac{d}{2}-\frac{d}{r}-\frac{1}{q}, \quad \mbox{ if $(q,r)$ is wave-admissible}.
}
}
By interpolation, we can obtain the Strichartz estimate for $(q,r)$ between wave-admissible and sharp Schr\"odinger-admissible.
\end{lemma}
\begin{remark}\label{re:H1/2}
In particular, $\left(2+\frac{4}{d-1}, 2+\frac{4}{d-1} \right)$ is sharp wave-admissible and $\left(2+\frac{4}{d}, 2+\frac{4}{d} \right)$ is sharp Schr\"odinger-admissible. We have for $d\geq 2$,
\begin{align}\label{eq2.8v161}
\normo{e^{it\jb{\nabla}}f}_{L_{t,x}^{2+\frac{4}{d-1}}(\R\times \R^d)} + \normo{e^{it\jb{\nabla}}f}_{L_{t,x}^{2+\frac{4}{d}}(\R\times \R^d)}
\les  & \norm{f}_{H^{ \frac12}(\R^d)}.
\end{align}
\end{remark}
By the duality, we have
\begin{lemma}[Inhomogeneous Strichartz estimate]\label{le2.3}
Assume $v$ and $G$ satisfy the following equations on the time interval $I\subseteq \mathbb{R}$,
\begin{align*}
i \partial_t v - \langle \nabla \rangle v = \langle \nabla \rangle^{-1} G.
\end{align*}
Then
\begin{align*}
& \left\|\langle \nabla \rangle^{ 1 + \frac{d+2}2 \left( \frac1r - \frac12 \right)}
v \right\|_{L_t^q L_x^r(I\times \mathbb{R}^d)} \lesssim \left\|\langle \nabla \rangle v(t_0) \right\|_{L^2(\mathbb{R}^d)}
+ \left\|\langle \nabla \rangle^{\frac{d+2}2 \left( \frac12 - \frac1{\tilde{r}} \right)} G \right\|_{L_t^{\tilde{q}' } L_x^{\tilde{r}'}(I \times \mathbb{R}^d)}
\end{align*}
for each $t_0 \in I$ and any sharp Schr\"odinger-admissible pairs $(q,r)$ and $\left(\tilde{q}, \tilde{r}\right)$.
\end{lemma}
For the low frequency component, the Klein-Gordon propagator behaves like Schr\"odinger equation.  By applying the bilinear restriction estimate of \cite{T1} as in \cite{KSV1} (see \cite{Killip-Visan1}), we obtain the following refined Strichartz estimate which is the same as the Schr\"odinger equation.
\begin{lemma}[Refined Strichartz] \label{co4.8}
$\forall\, f \in L_x^2(\mathbb{R}^d)$ and $supp \hat{f} \subseteq \left\{ |\xi | \le 2^d \right\}$, we have
\EQ{\label{eq2.8}
\left\|e^{-it \langle \nabla \rangle} f \right\|_{L_{t,x}^\frac{2(d+2)}d(\mathbb{R}\times \mathbb{R}^d)} \les& \|f\|_{L_x^2}^\frac{d+1}{d+2} \left( \sup\limits_{ \mathcal{C} } | \mathcal{C} |^{-\frac{d+1}{2(d^2 + 3d + 1)}} \left\|e^{-it\langle \nabla \rangle} P_{\mathcal{C}}  f \right\|_{L_{t,x}^\frac{2 \left(d^2 + 3d + 1 \right)}{d^2} (\mathbb{R}\times \mathbb{R}^d)} \right)^\frac1{d+2}.
}
where the supremum is taken over all dyadic cubes $\mathcal{C}$ with side length no more than $2^{d+1}$, and $P_{\mathcal{C}} f$ is the Fourier restriction of $f$ to $ \mathcal{C}$.

As a consequence, by interpolation, 
we obtain
\begin{align}\label{eq2.11v165}
\left\|e^{-it \langle \nabla \rangle} f \right\|^{\frac{d^2+2d+1}{d^2+3d+1}}_{L_{t,x}^\frac{2(d+2)}d(\mathbb{R}\times \mathbb{R}^d)} \les& \|f\|_{L_x^2}^\frac{d+1}{d+2} \left( \sup\limits_{ \mathcal{C} } | \mathcal{C} |^{-\frac{d+1}{2(d^2 + 3d + 1)}} \left\|e^{-it\langle \nabla \rangle} P_{\mathcal{C}}  f \right\|_{L_{t,x}^\infty}^{\frac{d+1}{d^2+3d+1}} \right)^\frac1{d+2}.
\end{align}
\end{lemma}
\begin{remark}\label{rem:improve}
In the two dimensional case \cite{KSV1}, the combination of
the cube decomposition \eqref{eq2.8}
and a tube-type decomposition is used to obtain the inverse Strichartz estimate.
It will turn out that the decomposition \eqref{eq2.11v165} is sufficient for this purpose.
\end{remark}
By the Strichartz estimate and Picard's iteration, we can establish the well-posedness theory for \eqref{eq2.2}.
\begin{proposition}[Local well-posedness in $H^1$]\label{pr3.1}
For any $v_0 \in H_x^1(\mathbb{R}^d)$, there exists a unique maximal-lifespan solution $v: I \times \mathbb{R}^d \to \mathbb{C}$ to \eqref{eq2.2} with $v(0) = v_0$ satisfying $S_J(v)<\infty$ for any $J\Subset I$. Moreover, we have
\begin{enumerate}
\item $|I|\geq C(\norm{v_0}_{H^1})$. If $I=\R$ and $S_{\mathbb{R} }(v) < \infty$,  then $v$ scatters in $H^1$.

\item If $\|v_0\|_{H_x^1}$ is small enough, then $I=\R$ and $S_{\mathbb{R}}(v) \lesssim 1$.

\item If $J\subseteq I$ and $S_J(v) < L$, then for any $0 \le s <1 + \frac4d$, we have
\begin{align}\label{eq3.3}
\left\|\langle \nabla \rangle^{s  +  \frac{d+2}2 \left( \frac1r - \frac12 \right)} v\right\|_{L_t^q L_x^r(J\times \mathbb{R}^d)} \lesssim \left\|\langle \nabla \rangle^{s } v_0 \right\|_{L_x^2},
\end{align}
where $(q,r)$ is sharp Schr\"odinger-admissible.
\end{enumerate}
\end{proposition}
In the defocusing case we can extend the local well-posedness to global well-posedness by the energy conservation. For the focusing case, we have global well-posedness under the restriction $E(u_0,u_1)< E(Q,0)$ and $\norm{u_0}_2 <  \norm{Q}_2$ (see the next subsection).  To prove the scattering, we need the following stability theorem which can be proved by the Strichartz estimates.  This theorem  is used in the proof of Theorem \ref{th6.2}(the approximation of the large scale profile) and Theorem \ref{th1.9}(the existence of the critical element).
\begin{proposition}[Stability theorem]\label{pr3.4}
Assume $\tilde{v}$ solves
\begin{align*}
i\tilde{v}_t - \langle \nabla \rangle \tilde{v} = \mu \langle \nabla \rangle^{-1} \left(|\Re \tilde{v}|^\frac4d \Re \tilde{v}\right)  + e_1 + e_2 +e_3
\end{align*}
on the time interval $I\subseteq \R$ with error terms $e_1$, $e_2$ and $e_3$, and satisfies
$\big\|\langle \nabla \rangle^\frac12 \tilde{v} \big\|_{L_t^\infty L_x^2} \le M$ and $
\left\|\Re \tilde{v} \right\|_{L_{t,x}^\frac{2(d+2)}d(I \times \mathbb{R}^d)} \le L$ for some constants $M,L > 0$.
Assume further for some $t_0 \in I$ and
$\left\|\langle \nabla \rangle^\frac12 (v_0 - \tilde{v}(t_0)) \right\|_{L^2} \le M'$
for some constant $M'> 0$.  Then there exists $\e=\epsilon(M,M',L)>0$ with the following properties: if
\begin{align*}
& \left\|e^{-i(t-t_0) \langle \nabla \rangle} (v_0 - \tilde{v}(t_0)) \right\|_{L_{t,x}^\frac{2(d+2)}d(I \times \mathbb{R}^d)} \\
& \ + \left\|e_1 \right\|_{L_t^1 H_x^\frac12} + \left\|\langle \nabla \rangle e_2 \right\|_{L_{t,x}^\frac{2(d+2)}{d+4}(I\times \mathbb{R}^d)} + \left\|\int_{t_0}^t e^{-i(t-s) \langle \nabla \rangle} e_3(s) \,\mathrm{d}s \right\|_{L_{t,x}^\frac{2(d+2)}d \cap L_t^\infty H_x^\frac12(I \times \mathbb{R}^d) } \le \epsilon,
\end{align*}
then there exists a solution $v$ to \eqref{eq2.2} on the time interval $I$ with $v(t_0) = v_0$, and $v$ satisfies
\begin{align*}
\left\|v-\tilde{v} \right\|_{L_{t,x}^\frac{2(d+2)}d(I\times \mathbb{R}^d)} \le \epsilon C(M,M',L), \text{ and }
\left\|v- \tilde{v} \right\|_{L_t^\infty H_x^\frac12(I\times \mathbb{R}^d)} \le M' C(M,M',L).
\end{align*}
\end{proposition}

\subsection{Variational estimate}
In this subsection, we collect some variational estimates which are needed when studying the focusing NLKG. The variational estimates are known in \cite{IMN} or can be proved using similar arguments.  See also \cite{JL,Z}.
For $( \alpha, \beta) \in \R^2$,
let
\EQ{
m_{\alpha,\beta}:= \inf\left\{ E(\varphi, 0): \varphi \in H^1(\mathbb{R}^d)\setminus\{0\}, \mathcal{K}_{\alpha,\beta} (\varphi) = 0\right\},
}
where
\begin{align*}
\mathcal{K}_{\alpha, \beta} (\varphi) = & \frac{\partial}{\partial \lambda}\Big|_{\lambda=0} E \left(e^{\alpha \lambda}\varphi \left(e^{\beta \lambda}x \right),0 \right)\\
=& \int_{\mathbb{R}^d} \frac{2 \alpha -(d-2) \beta}2 |\nabla \varphi|^2 + \frac{2 \alpha - d \beta}2 |\varphi|^2  - \left( \alpha -\frac{d^2 \beta}{ 2(d+2)} \right) |\varphi|^{\frac{2(d+2)}d  } \,\mathrm{d}x.
\end{align*}
Let
\begin{align*}
\mathcal{K}_{\alpha, \beta}^+ = \left\{ (u_0,u_1) \in H^1\times L^2: E(u_0,u_1) < m_{\alpha, \beta}, \mathcal{K}_{\alpha, \beta}(u_0) \ge 0\right\},\\
\mathcal{K}_{\alpha, \beta}^- = \left\{ (u_0,u_1) \in H^1\times L^2: E(u_0,u_1) < m_{\alpha, \beta}, \mathcal{K}_{\alpha, \beta}(u_0)  < 0\right\}.
\end{align*}
In particular, we will use
\begin{align*}
\mathcal{K}_{0}(\varphi) : = \mathcal{K}_{1,0}(\varphi),
\text{ and }
\mathcal{K}_1(\varphi) : = \mathcal{K}_{d, 2}(\varphi).
\end{align*}
$\mathcal{K}_{0}(\varphi)$ is convenient for the blow-up while $\mathcal{K}_1(\varphi)$ is convenient for the scattering. As a sign-functional, they play the same roles.
\begin{lemma}\label{le2.4v31}
We have $m_{1,0} =m_{d,2}= E(Q,0)> 0 $, where $Q \in H^1$ is the ground state of \eqref{eq1.4}.  Moreover,
$\mathcal{K}_{1, 0}^\pm=\mathcal{K}_{d, 2}^\pm$.
\end{lemma}
\begin{remark}
Recall that $Q$ is the unique (up to symmetry) maximizer to the following sharp Gagliardo-Nirenberg inequality:
\begin{align}\label{eq2.13v178}
 \left\|f\right\|_{L_x^\frac{2(d+2)}d}^\frac{2(d+2)}d \le \frac{d+2}d \left( \frac{\|f\|_{L_x^2}}{\|Q\|_{L_x^2}}\right)^\frac4d
  \|\nabla f\|_{L_x^2}^2.
\end{align}
As a result, we have $\mathcal{K}_{1, 0}^- = \mathcal{K}_{d, 2}^- = \left\{E(u_0,u_1) < E(Q,0): \norm{u_0}_{L^2}>\norm{Q}_{L^2}\right\}$.
\end{remark}
\begin{proposition}\label{pr2.12}
Let $u: I \times \mathbb{R}^d \to \mathbb{R}$ be a solution with maximal lifespan $I= \left(-T_*, T^* \right)$ to \eqref{eq1.1} with $(u(0), u_t(0)) \in H_x^1 \times L_x^2$,  $\mu=-1$,  and $E(u(0), u_t(0)) < E(Q, 0)$.

\begin{itemize}
\item If $\norm{u_0}_2 < \norm{Q}_2$, then $I=\R$ and for any $t \in \mathbb{R} $,
\begin{align}
E\left(u(t), \partial_t u(t) \right)
\le& \frac12 \int_{\mathbb{R}^d}  |\nabla u |^2 + |u|^2 + |\partial_t u |^2   \,\mathrm{d}x
 \le \left( 1 + \frac{d}2\right) E\left(u, \partial_t u \right), \label{eq2.34} \\
\mathcal{K}_0(u(t) ) \ge \ & c \min\left( E(Q,0) - E\left(u(0),u_t(0) \right), \left\|u(t) \right\|_{H^1_x}^2  \right), \notag \\
\mathcal{K}_1(u(t) ) \ge \ & c \min\left( E(Q,0) - E\left(u(0), u_t(0) \right) , \left\|\nabla u(t) \right\|_{L_x^2}^2 \right). \notag
\end{align}

\item If $\norm{u_0}_2> \norm{Q}_2$, then $\max \left(T_*,T^* \right)<\infty$.
Moreover, we have for any $t\in I$,
\begin{align*}
\mathcal{K}_0(u(t) ) \le& - 2 \left( E(Q,0) - E\left(u(0), u_t(0) \right) \right) < 0,
\intertext{ and }
\mathcal{K}_1(u(t)) \le  & - 2 \left( E(Q,0) - E(u(0), u_t(0) ) \right) < 0.
\end{align*}
\end{itemize}
\end{proposition}
\begin{remark}
The blow-up part in the above proposition can be proven by showing the strict concavity of $\left\|u(t) \right\|_{L_x^2}^{-\frac2d}$. We will omit the details of the proof but refer to \cite{IMN,KSV1,NS2,PS} for similar argument.
\end{remark}

\section{Proof of the main theorem}\label{se7}
In this section, we prove the main theorem, assuming two crucial ingredients which will be proved in the remaining sections.
The proof follows closely Kenig-Merle's road map and the ideas in \cite{KSV1}.  Let
\begin{equation*}
\Lambda (E) = \sup   \| u \|_{L_{t,x}^{ \frac{2(d+2)}d }   (\mathbb{R} \times \mathbb{R}^d )},
\end{equation*}
where the supremum is taken over all solutions $u \in C_t^0 H_{x }^{ 1}$ of \eqref{eq1.1} obeying $ {E}(u,\partial_t u) \le E$ (and an extra assumption $\norm{u_0}_{L^2}<\norm{Q}_{L^2}$ when $\mu = -1$).

Let $E_{c} = \sup \{ E: \Lambda  (E) < \infty\}$. We have $E_c > 0$ by the small data scattering results in Proposition \ref{pr3.1}.
To prove Theorem \ref{th1.4}, we only need to show $E_{c}  = \infty $(when $\mu = 1$) and $E_c = E(Q,0) $(when $\mu = - 1$).  We will prove it by contradiction argument.

\subsection{Existence of critical element}
The main result of this subsection is
\begin{theorem}[Existence of an critical element]\label{th1.9}
Assume $E_{c}  <  \infty $(when $\mu = 1$) and $E_c <  E(Q, 0) $(when $\mu = - 1$). There exists a global solution $u_c$ to \eqref{eq1.1} with $E\left(u_c, \partial_t u_c \right) = E_c$ (and also $\norm{u_c(0)}_{L^2}<\norm{Q}_{L^2}$ when $\mu=-1$) and
$\|u_c\|_{L_{t,x}^\frac{2(d+2)}d ( \R \times \mathbb{R}^d )} = \infty$.
Furthermore, there exists $x: \mathbb{R} \to \mathbb{R}^d$
such that
\begin{align}\label{eq3.1v182}
\left\{ (u_c, \partial_t u_c)(t, x+ x(t)): t\in \mathbb{R} \right\}
\end{align}
 is pre-compact in $H^1 \times L^2$.
%
\end{theorem}
As a direct consequence of the pre-compactness of the critical element, we have
\begin{corollary}
For any $\eta >0$, there is $C: \mathbb{R}^+ \to \mathbb{R}^+$ such that
\begin{align}\label{eq3.1v161}
& \int_{|x- x(t)| \ge  C(\eta)} \left| \nabla  u_c(t,x)\right|^2 + \left|   u_c(t,x)\right|^2+ \left|\partial_t u_c(t,x)\right|^2 + \left|u_c(t,x)\right|^\frac{2(d+2)}d\, \mathrm{d}x \notag \\
& \quad + \int_{| \xi|  \le \frac1{ C(\eta)}} \left|  \widehat{u_c}(t, \xi)\right|^2
 \,\mathrm{d} \xi < \eta, \ \forall\, t \in \mathbb{R}.
\end{align}
\end{corollary}
The proof of the above theorem relies on two ingredients.  The first one is the profile decomposition associated to the linear Klein-Gordon equation with data in $H^1$ at the $L^2$-critical level.  More precisely, due to the mass-criticality, we need to understand the defect of compactness of the following Strichartz estimate
\EQ{
\normo{e^{it\jb{\nabla}}f}_{L_{t,x}^{2+\frac{4}{d}}(\R\times \R^d)}\les& \norm{f}_{H^1}.
}
There are several non-compact groups of symmetry in the above inequality.
The first one is the spatial translation:
\EQ{f\to \left(T_y f \right)(\cdot) : = f(\cdot-y), \quad y\in \R^d.}
The second one is phase modulation:
\EQ{f\to  e^{i\theta \langle{\nabla}\rangle}f, \quad \theta \in \R.}
The third one is dilation in one direction (not a group):
\EQ{f\to  D_\lambda f:=\lambda^{-\frac{d}2}f\left( \frac{\cdot}\lambda \right), \quad \lambda \in [1,\infty).}
For $f\in H^1$, we see $\normo{D_\lambda f}_{L^2}=\norm{f}_{L^2}$ and $\normo{D_\lambda f}_{\dot H^1}\to 0$ as $\lambda \to \infty$. We introduce a Schr\"odinger dilation
\EQ{
u(t,x)\to \widetilde D_\lambda u:=\lambda^{- \frac{d}2}u \left( \frac{t}{\lambda^2}, \frac{x}\lambda \right).
}
We have $\normo{\widetilde D_\lambda u}_{L_t^qL_x^r}=\norm{u}_{L_t^q L_x^r}$ when $(q,r)$ is sharp Sch\"odinger-admissible.
Moreover,
\EQ{
\widetilde D_{\lambda^{-1}} \left[e^{it \langle{\nabla}\rangle} D_\lambda f \right] = e^{i\lambda^2 t \langle{\lambda^{-1}\nabla}\rangle} f.
}
The last one is the Lorentz transformation.  We use the version as in \cite{KSV1}. For any $\nu \in \mathbb{R}^d$, we define the Lorentz boost of the space-time:
\begin{align*}
\left(\tilde{t}, \tilde{x}\right) = L_\nu(t,x) : = \left(\langle \nu \rangle t - \nu \cdot x, x^\perp + \langle \nu \rangle x^\parallel - \nu t\right),
\end{align*}
where $x^\perp= x- \frac{( x \cdot \nu) \nu}{|\nu|^2}$ and $x^\parallel= \frac{ (x\cdot \nu) \nu}{ |\nu|^2} $. An easy computation yields
\EQ{
(t,x)=L_\nu^{-1}  \left(\tilde t, \tilde x \right)= L_{-\nu} \left(\tilde t, \tilde x \right)
= \left(\langle \nu \rangle \tilde t + \nu \cdot \tilde x, {\tilde x}^\perp + \langle \nu \rangle {\tilde x}^\parallel + \nu \tilde t\right)
}
and that if $u(t,x) = e^{-i t \langle \xi\rangle  + i x \cdot \xi}$, then
\begin{align}\label{eq2.9}
u \circ  L_\nu^{-1}\left(\tilde{t}, \tilde{x}\right) = e^{-i \tilde{t}\left\langle \tilde{\xi} \right\rangle  + i  \tilde{x} \cdot \tilde{\xi} },
\end{align}
where
\begin{align*}
\tilde{\xi} = { {l}}_\nu(\xi) : = \xi^\perp + \langle \nu \rangle \xi^\parallel - \nu \langle \xi \rangle.
\end{align*}
The action of the Lorentz boosts on the function $f$ is defined to be
\begin{align*}
\L_\nu f(x) : = \left( e^{-i \cdot \langle \nabla \rangle} f \right) \circ L_\nu (0,x).
\end{align*}
Then
\begin{align}\label{eq2.13v139}
\left(e^{-it\langle \nabla \rangle} \L_\nu^{-1} f\right)(x) : = \left(e^{-i \cdot \langle \nabla \rangle} f\right) \circ L_\nu^{-1}(t,x),
\end{align}
namely
$u $ is a solution of the linear Klein-Gordon equation if and only if $u\circ L_\nu^{-1}$ is a solution of the linear Klein-Gordon equation
with the following transformation on the initial data
\EQ{
f\mapsto
\left(\L_\nu f\right)(x) : = \left(e^{-i \cdot \langle \nabla \rangle } f\right) \circ L_\nu \left(0,x\right).
}
Moreover, direct computation yields (see \cite{KSV1})
\EQ{
\widehat{ \L_\nu^{-1} f }  \left(\tilde{\xi}\right) = \jb{\xi} \jb{\tilde{\xi}}^{-1} \hat{f}(\xi),\quad d\xi=\jb{\xi} \jb{\tilde{\xi}}^{-1} d\tilde{\xi},
}
and
\begin{align}\label{eq3.12v161}
\L_\nu^{-1} T_y e^{i\tau \langle \nabla \rangle} = T_{\tilde{y}} e^{i\tilde{\tau} \langle \nabla \rangle } \L_\nu^{-1}, \text{ where } \left(\tilde{\tau}, \tilde{y}\right) = L_\nu (\tau, y).
\end{align}
Furthermore, we have for any $s\in \mathbb{R}$,
\begin{align}\label{eq3.14v162}
\left\langle \L_\nu^{-1} f , g \right\rangle_{H^s} = \left\langle f, m_s^\nu(\nabla) \L_\nu g\right\rangle_{H^s}, \text{ with } m_s^\nu(\xi)
:= \left( \frac{\langle {\xi}\rangle}{\langle \tilde{\xi}\rangle} \right)^{1- 2s }
\end{align}
and $\left\|m_s^\nu \right\|_{L_\xi^\infty} + \left\|(m_s^\nu)^{-1} \right\|_{L_\xi^\infty} \lesssim \langle \nu \rangle^{|2s - 1|}$.

We can see the Lorentz transformation is unitary in $H^{\frac12}$, so assuming uniform boundedness in $H^1$ will require boundedness of $|\nu|$.  However, due to the mass-criticality, there is a $L^2$-dilation.  Thus the Lorentz transformation should also be taken into count for the defect of compactness.  More precisely, for $\nu_n\to \nu$ and $\lambda_n\to \infty$, as $n \to \infty$, we have for $f\in H^1$,
\EQ{
D_{\lambda_n} \L_{\nu_n}  f-\L_{\nu_n } D_{\lambda_n} f
}
is not small in $L^2$ in general.
By a direct calculation, we have
\begin{lemma}\label{le2.8}
For any $f \in L^2\setminus \{0\}$ and $\Lambda > 0$,
\begin{align*}
\mathcal{K} : = \left\{ D_\lambda^{-1} \L_\nu^{-1} m_0^\nu(\nabla)^{-1} e^{i\nu x} D_\lambda f: |\nu|\le \Lambda,
 \text{ and }
\Lambda^{-1} \le \lambda < \infty \right\}
\end{align*}
is a precompact subset of $L^2$, and $0 \notin \overline{\mathcal{K}}$. Furthermore, if $\hat{f}= \chi_{[-1,1]^d}$, we see for any $ R > 0$,
\begin{align}\label{eq2.20}
supp \, \hat{g} \subseteq \left\{ |\xi|\lesssim  \langle \Lambda \rangle \right\}, \ \|g\|_{L^2_x} \gtrsim \langle \Lambda
\rangle^{-1},
\text{ and } \int_{|x| \sim R} |g(x)|^2 \,\mathrm{d}x \lesssim \frac{\langle \Lambda \rangle}{\langle R\rangle},
\end{align}
uniformly for any $g \in \mathcal{K}$.
\end{lemma}
This result is the higher dimension extension of Lemma 2.8 in \cite{KSV1}, since the proof follows the similar argument, we omit the proof here.
We remark that $\mathcal{  K }$  is not compact, and the elements in  $\bar{\mathcal{K}} \setminus \mathcal{K} $ are characterized as follows:
$ h  \in \bar{\mathcal{ K} } \setminus \mathcal{ K } $ if and only if
$\hat{h} \left(\tilde{ \xi } \right) = \hat{f } \left( \tilde{\xi}^\perp + \langle \nu \rangle \tilde{\xi}^\parallel \right)$.

With all the symmetries above, we have
\begin{theorem}[Profile decomposition] \label{th5.1}
Assume $ \left\{ v_n \right\}_{n\ge 1} $ is a bounded sequence in $H_x^1(\mathbb{R}^d)$. Then, up to a subsequence, there exists $J_0 \in [1, \infty]$ and for each integer $1\leq j<J_0$, there also exist a non-zero function $ \phi^j \in L_x^2(\mathbb{R}^d)$, a sequence $\left\{\left(\lambda_n^j,t_n^j, x_n^j, \nu_n^j\right) \right\} \subseteq [ 1, \infty)\times \mathbb{R} \times \mathbb{R}^d \times \mathbb{R}^d$ with the following properties:
\begin{itemize}
\item either $\lambda_n^j \to \infty$ as $n\to \infty$ or $\lambda_n^j \equiv 1$;

\item either $\frac{t_n^j}{ \left(\lambda_n^j \right)^2} \to \pm \infty$ as $n\to \infty$ or $t_n^j \equiv 0$;

\item $\phi^j \in H^1$ if $\lambda_n^j \equiv 1$;

\item $\nu_n^j \to \exists \nu^j\in \mathbb{R}^d$ as $n\to \infty$, and $\nu_n^j \equiv 0$ if $\lambda_n^j \equiv 1$.
\end{itemize}
Let $P_n^j$ be the projector defined by
\begin{align}\label{eq3.17v182}
P_n^j \phi^j : =
\begin{cases}
\phi^j ,         & \text{ if } \lambda_n^j\equiv 1,\\
P_{\le (\lambda_n^j)^\theta }  \phi^j, &     \text{ if } \lambda_n^j \to \infty,
\end{cases}
\end{align}
where $0 < \theta \ll 1$.
For any $ 1\leq J <J_0 $, we have the decomposition
\begin{align}\label{eq3.17v173}
v_n = \sum\limits_{j=1}^J T_{x_n^j} e^{it_n^j \langle \nabla \rangle} \L_{\nu_n^j} D_{\lambda_n^j} P_n^j \phi^j + w_n^J,
\end{align}
with the decoupling properties
\begin{align}
 \left\|v_n \right\|_{L_x^2}^2
- \sum\limits_{j=1}^J \left\|T_{x_n^j} e^{it_n^j\langle \nabla \rangle} \L_{\nu_n^j} D_{\lambda_n^j} P_n^j \phi^j \right\|_{L_x^2}^2
 - \left\|w_n^J \right\|_{L_x^2}^2 \to 0,\label{eq3.16v162}\\
\left\|v_n\right\|_{H_x^1}^2
- \sum\limits_{j=1}^J \left\|T_{x_n^j} e^{it_n^j\langle \nabla \rangle} \L_{\nu_n^j} D_{\lambda_n^j} P_n^j \phi^j \right\|_{H_x^1}^2
- \left\|w_n^J \right\|_{H_x^1}^2 \to 0, \label{eq3.17v162}\\
 E(v_n) - \sum\limits_{j=1}^J E\left(T_{x_n^j} e^{it_n^j \langle \nabla \rangle} \L_{\nu_n^j} D_{\lambda_n^j} P_n^j \phi^j\right) - E(w_n^J) \to 0, &\text{ as } n\to \infty,\label{eq3.1v141}
\end{align}
and
\begin{align}\label{eq3.21v173}
\limsup\limits_{n\to \infty} \Big\|e^{-it\langle \nabla \rangle} w_n^J \Big\|_{L_{t,x}^\frac{2(d+2)}d(\mathbb{R}\times \mathbb{R}^d)}  \to 0, \text{ as } J \to J_0.
\end{align}
Moreover, for any $j\ne j'$, the orthogonality relation
\begin{align}\label{eq3.21v172}
\frac{\lambda_n^j}{\lambda_n^{j'}} + \frac{\lambda_n^{j'}}{\lambda_n^j} + \lambda_n^j \big|\nu_n^j - \nu_n^{j'}  \big| + \frac{ \big|s_n^{jj'} \big|}{ \big(\lambda_n^{j'} \big)^2} + \frac{ \big|y_n^{jj'} \big|} {\lambda_n^{j'}} \to \infty, \quad \text{ as } n \to \infty,
\end{align}
holds, where $\left(- s_n^{jj'}, y_n^{jj'}\right) : = L_{\nu_n^{j'}}\left(t_n^{j'} - t_n^j, x_n^{j'} - x_n^j\right)$.
\end{theorem}
The second ingredient in the proof of Theorem \ref{th1.9} is the following NLS approximation.  We will use it to construct the nonlinear profile.
\begin{theorem}[Approximation of the low frequency profile]\label{th6.2}
Assume
$\nu_n \to \nu \in \mathbb{R}^d$, $\lambda_n \to \infty$, and either $t_n \equiv 0$ or $\frac{t_n}{\lambda_n^2} \to \pm \infty$, as $n\to \infty$. Let $\phi \in L_x^2(\mathbb{R}^d)$, and also assume
\begin{equation}\label{eq:masscond}
\|\phi\|_{L^2} < \left({2C_d}\right)^{-\frac{d}4} \|Q\|_{L^2}, \text{ if } \mu = - 1.
\end{equation}
Let
\begin{align*}
\phi_n : = T_{x_n} e^{it_n \langle \nabla \rangle} \L_{\nu_n} D_{\lambda_n} P_{\le \lambda_n^\theta} \phi,
\end{align*}
where $\theta$ is some sufficiently small positive number. There exists a global solution $v_n$ of \eqref{eq2.2} with $v_n(0) = \phi_n$ for $n$ large enough satisfying
\begin{align*}
S_{\mathbb{R}}(v_n) \lesssim_{\|\phi\|_{L^2} } 1.
\end{align*}
Moreover, $\forall\, \epsilon > 0$, there exist $N_\epsilon > 0$ and $\psi_\epsilon \in C_c^\infty(\mathbb{R} \times \mathbb{R}^d)$ so that for each $n > N_\epsilon$, we have
\begin{align}\label{eq6.2}
\left\|\Re\left( v_n \circ L_{\nu_n}^{-1} \left(t + \tilde{t}_n, x + \tilde{x}_n\right)
- {\lambda_n^{- \frac{d}2 } }{e^{-it}} \psi_\epsilon \left( \frac{t}{\lambda_n^2}, \frac{x}{\lambda_n}\right) \right) \right\|_{L_{t,x}^\frac{2(d+2)}d(\mathbb{R} \times \mathbb{R}^d)} < \epsilon,
\end{align}
where $(\tilde{t}_n, \tilde{x}_n ) : = L_{\nu_n}(t_n, x_n)$.
\end{theorem}
If we assume the above two theorems hold, then by Proposition \ref{pr2.12}, Theorem \ref{th5.1}, Theorem \ref{th6.2}, and Proposition \ref{pr3.4}, with similar argument as in \cite{IMN,KSV1}, we can give the following result.
\begin{proposition}[P.S. condition modulo translations] \label{pr7.1}
Let $u_n$ be a sequence of global solutions to \eqref{eq1.1}, which satisfy
\begin{align}
\lim\limits_{n\to \infty} S_{(- \infty, 0]} (u_n) = \lim\limits_{n\to \infty} S_{[ 0, \infty) } (u_n) = \infty, \notag\\
\left\|u_n(0)\right\|_{L^2} < \| Q\|_{L^2}, \text{ when $\mu = -1$, } \label{eq7.2} \\
\intertext{ and also }
E(u_n) \nearrow E_c \quad  \text{ as } n\to \infty. \notag
\end{align}
Then $\left(u_n(0), \partial_t u_n(0)\right)$ converges in $H^1\times L^2$ modulo translations up to a subsequence.
\end{proposition}
The proposition can be shown in the same spirit as in \cite{KSV1}.
Let us give a brief outline of the proof to see how the tools we have developed by now are used.
\begin{proof}[Outline of the proof]
Let
\begin{align*}
v_n : = u_n + i \langle \nabla \rangle^{-1} \partial_t u_n,
\end{align*}
and we will show $v_n(0)$ converges in $H^1$ modulo translations after passing to a subsequence. When $\mu = -1$, by Proposition \ref{pr2.12} and \eqref{eq7.2}, we have $v_n$ satisfies
\begin{align*}
\left\|v_n(0) \right\|_{L^2}^2  \le 2 E_c < \|Q\|_{L^2}^2.
\end{align*}
Thus, for both defocusing and focusing cases, we get
\begin{align*}
\left\|v_n(0) \right\|_{H^1}^2 \lesssim E(v_n) \le E_c.
\end{align*}
We can then apply Theorem \ref{th5.1} to the sequence $v_n(0)$, and have for any $J \in [1, J_0) \cap \mathbb{N}$,
\begin{align*}
v_n(0) = \sum\limits_{j =1}^J \phi_n^j + w_n^J,
\end{align*}
with
\begin{align*}
\phi_n^j = T_{x_n^j} e^{it_n^j \langle \nabla \rangle} \L_{\nu_n^j} D_{\lambda_n^j} P_n^j \phi^j.
\end{align*}
For any $ 1 \le j \le J_0$, we can make sure that $\left\|\phi_n^j\right\|_{L^2}$ and $E(\phi_n^j)$ converge after passing to a subsequence. By \eqref{eq3.1v141}, we also have
\begin{align}\label{eq7.8}
 E_c = \lim\limits_{n\to \infty} E(v_n) = \lim\limits_{n\to \infty} \left( \sum\limits_{j = 1}^J E\left(\phi_n^j \right) + E \left(w_n^J \right) \right).
\end{align}
In the sequel, let us restrict ourselves to the case $J_0=1$. The preclusion of the case $J_0\ge 2$ is standard, see for instance \cite{Ch,KSV1}. In this case, the identity
\begin{align}\label{eq7.9}
\lim\limits_{n \to \infty} E(\phi_n^1) = E_c
\end{align}
follows also by a standard argument. By \eqref{eq7.8} and \eqref{eq2.34}, we have
\begin{align}\label{eq7.12}
v_n - \phi_n^1 = w_n^1 \to 0 \text{ in } H_x^1, \quad \text{ as } n \to \infty.
\end{align}
We now divide the analysis into the following three cases.
\begin{itemize}
\item[Case 1.] $\lambda_n^1  = 1$ and $t_n^1 = 0$;
\item[Case 2.] $\lambda_n^1 = 1$ and $t_n^1 \to \pm \infty$;
\item[Case 3.] $\lambda_n^1 \to \infty$ as $n\to\infty$.
\end{itemize}
In the first case, we have the desired conclusion.
The second case is precluded by a standard argument, see for example \cite{Ch,KSV1}.
We omit the details.

Let us show that the third case can also be precluded. We will apply Theorem \ref{th6.2}, but when $\mu = -1$, we need to verify the following result first:
\begin{lemma}\label{le7.3}
When $\mu = - 1$, if $\lim\limits_{n\to \infty} \lambda_n^1 = \infty$, we have $\left\|\phi^1 \right\|_{L^2}< \| Q\|_{L^2} $.
\end{lemma}
\begin{proof}
Using \eqref{eq7.9} and \eqref{eq3.12v161}, we obtain
\begin{align*}
\langle \nu_\infty^1 \rangle \left\|\phi^1 \right\|_{L^2}^2
 =  & \lim\limits_{n\to \infty} \int_{\mathbb{R}^d} \left\langle
{\left( \lambda_n^1 \right)^{-1} \xi } \right\rangle \left\langle l_{- \nu_n^1} \left(
{\left( \lambda_n^1 \right)^{-1} \xi } \right) \right\rangle \left|  \left( \mathcal{F}\left(  P_{\le \left(\lambda_n^1 \right)^\theta}  \phi^1 \right) \right) (\xi) \right|^2 \,\mathrm{d}\xi \\
= &   \lim\limits_{n\to \infty} 2 E \left(\phi_n^1 \right) \le   2 E_c.
\end{align*}
This together with $2 E_c < 2 E(Q) =  \|Q\|_{L^2}^2$ implies the result.
\end{proof}
By Theorem \ref{th6.2}, $v_n^1$ with $v_n^1(0) = \phi_n^1$ is a global solution to \eqref{eq2.2} and $S_{\mathbb{R}} \left(v_n^1 \right) \lesssim_{E_c} 1$ for $n$ large enough.
Let us recall that the mass assumption of Theorem \ref{th6.2} in the focusing case is
\begin{equation}
\left\|\phi^1\right\|_{L^2}^2 < (2 C_d)^{-\frac{d}2} \|Q\|_{L^2}^2,
\end{equation}
which is fulfilled because $C_d < \frac12$. Using \eqref{eq7.12} and Proposition \ref{pr3.4}, we can conclude $S_{\mathbb{R}}(v_n) < \infty$, this is a contradiction and therefore completes the proof of Proposition \ref{pr7.1}.
\end{proof}
Once we get the (P.S.) condition modulo translation in Proposition \ref{pr7.1}, we can extract a special solution of NLKG, which is the critical element in Theorem \ref{th1.9}.

\subsection{Rigidity}
By a virial type argument, we can exclude the critical element, thus concluding the proof of Theorem \ref{th1.4} in the following theorem. We refer to \cite{IMN,KSV1} for a proof.
\begin{theorem}[Nonexistence of the critical element]\label{th8.1}
The critical element $u_c$ in Theorem \ref{th1.9} does not exist.
\end{theorem}
Before giving the proof of this theorem, we first collect some properties of the critical element. In  
 the defocusing case and also invoking \eqref{eq2.34} in the focusing case, we have
\begin{align}\label{eq8.1}
\|u_c\|_{L_t^\infty H_x^1}^2 + \|\partial_t u_c \|_{L_t^\infty L_x^2}^2 \le 4 E(u_c ).
\end{align}
By the Lorentz invariance of the NLKG and the minimality of $u_c$ as a blow-up solution, we have 
\begin{align}\label{eq8.2}
P(u_c , \partial_t u_c ) = 0.
\end{align}
This leads to the control of $x(t)$ in \eqref{eq3.1v182}.
\begin{lemma}[Controlling $x(t)$] \label{le8.2}
The spatial center function $x(t)$ of $u_c$ satisfies
\begin{align}\label{eq3.34v182}
\left|\frac{x(t)}{t} \right| = 0,  \text{
as $|t| \to \infty$.}
\end{align}
\end{lemma}
\begin{proof}
By the spatial translation invariance, we may assume that $x(0) = 0$. Suppose \eqref{eq3.34v182} is not true,
there would exist $\delta > 0$ and a sequence $t_n \to \pm \infty$ such that
\begin{align*}
|x(t_n) | > \delta |t_n|.
\end{align*}
We may assume that $t_n \to \infty$.
Let $\eta \ll E(u_c)$ be a sufficiently small positive constant
and define
\begin{align*}
R_n : = C(\eta) + |x(t_n)|.
\end{align*}
Let $\psi:\mathbb{R}_+ \to [0, 1]$ be a smooth cut-off function with
\begin{align}\label{eq3.35v182}
\psi(r) =
\begin{cases}
1, \text{  $ 0 \le r <  1$,} \\
 0, \text{  $r >  2$}
\end{cases}
\end{align}
 and define 
\begin{align*}
X_{R_n}(t) : = \int_{\mathbb{R}^d} x \psi \left( \frac{|x|}{R_n} \right) e_{u_c} (t,x) \,\mathrm{d}x,
\end{align*}
where
\begin{align*}
e_{u_c} : = \frac12 |u_c |^2 + \frac12 |\nabla u_c |^2 + \frac12 |\partial_t u_c |^2 +  \mu \frac{d}{2(d+2)} |u_c |^\frac{2(d+2)}d.
\end{align*}
By the triangle inequality, \eqref{eq8.1} and \eqref{eq3.1v161}, we have
\begin{align*}
|X_{R_n}(0)| \le \int_{|x| \le C(\eta)} |x| |e_{u_c} (0,x)| \,\mathrm{d}x + \int_{C(\eta) \le |x| \le 2 R_n} |x| |e_{u_c} (0,x)| \,\mathrm{d}x \lesssim C(\eta) E(u_c ) + \eta R_n.
\end{align*}
By the triangle inequality and \eqref{eq3.1v161}, 
we also have
\begin{align*}
|X_{R_n}(t_n)| & \ge |x(t_n)| E(u_c ) - \int_{|x- x(t_n)| \le C(\eta)} | x - x(t_n)| \psi\left( \frac{|x|}{R_n} \right) |e_{u_c} (t_n)| \,\mathrm{d}x \\
&\quad  - \int_{|x- x(t_n)| \ge C(\eta)} | x- x(t_n)|  \psi \left( \frac{ |x|}{R_n} \right) |e_{u_c} (t_n)| \,\mathrm{d}x
- |x(t_n)| \int_{\mathbb{R}^d} \left( 1- \psi\left( \frac{ |x|}{R_n} \right) \right) |e_{u_c} (t_n)| \,\mathrm{d}x \\
& \ge |x(t_n)| \left(E(u_c) - 4 \eta \right) - C(\eta) \left( 2 E(u_c) + 2 \eta \right).
\end{align*}
Thus, we get
\begin{align}\label{eq8.4}
|X_{R_n}(t_n) - X_{R_n}(0)| \gtrsim_{E(u_c )} |x(t_n)| - C(\eta).
\end{align}
By a direct calculation relying on \eqref{eq8.2}, we have 
\begin{align*}
 X_{R_n}'(t) = \int_{\mathbb{R}^d } \left( 1 - \psi\left( \frac{|x|}{R_n} \right) \right) \partial_t u_c \nabla u_c  \,\mathrm{d}x
- \int_{\mathbb{R}^d} \frac{x}{|x|R_n} \psi'\left( \frac{|x|}{R_n} \right) \partial_t u_c \,  x \cdot \nabla u_c  \,\mathrm{d}x.
\end{align*}
This together with \eqref{eq3.1v161} yields
\begin{align}\label{eq8.5}
|X_{R_n}'(t)| \lesssim \eta.
\end{align}
We can now derive an estimate by \eqref{eq8.4} and \eqref{eq8.5}, that is  
\begin{align*}
\eta t_n \gtrsim |X_{R_n}(t_n) - X_{R_n}(0) | \gtrsim_{E(u_c )} |x(t_n)| - C(\eta) \gtrsim_{E(u_c )} \delta t_n - C(\eta),
\end{align*}
once taking $\eta$ small enough depending on $\delta$ and $E(u_c )$, and then taking $n$ sufficiently large, we arrive a contradiction. Therefore, we get \eqref{eq3.34v182}.
\end{proof}
We now turn to the proof of Theorem \ref{th8.1}.
Let $\eta_1 $ and $\eta_2 $ ba small positive constants to be determined later. By Lemma \ref{le8.2}, there exists $T_0 = T_0(\eta_1)>0$ such that
\begin{align}\label{eq8.6}
|x(t)| \le \eta_1 t \text{ for any } t \ge T_0.
\end{align}
By Plancherel's identity and \eqref{eq3.1v161}, we have
\begin{align}\label{eq8.9}
\int_{\mathbb{R}^d } |u_c (t,x)|^2 \,\mathrm{d}x\le \eta_2 + C(\eta_2)^2 \int_{\mathbb{R}^d} |\nabla u_c (t, x)|^2 \,\mathrm{d}x.
\end{align}
With $\psi$ defined as in \eqref{eq3.35v182} 
and $0 < \epsilon < 1 < R$ to be specified later, we define
\begin{align*}
Z_R(t) = - \int_{\mathbb{R}^d} \psi\left( \frac{|x|}{R} \right) \partial_t u_c (t,x) x \cdot \nabla u_c (t,x) \,\mathrm{d}x
- (1 - \epsilon) \int_{\mathbb{R}^d} \partial_t u_c (t,x) u_c (t,x) \,\mathrm{d}x.
\end{align*}
By the Cauchy-Schwarz inequality and \eqref{eq8.1}, we have
\begin{align}\label{eq8.10}
|Z_R(t)| \lesssim R E(u_c ) \lesssim_{u_c} R.
\end{align}
On the other hand, by direct calculation, we have 
\begin{align*}
Z_R'(t) =  & \epsilon \left( \|u_c (t) \|_{H^1}^2 + \|\partial_t u_c \|_{L^2}^2 \right)
 + (1 - 2\epsilon) \int_{\mathbb{R}^d} |\nabla u_c (t)|^2 +  \mu  \frac{d}{ d+2}  |u_c (t)|^\frac{2(d+2)}d \,\mathrm{d}x\\
& - 2\epsilon \int_{\mathbb{R}^d} |u_c (t)|^2 \,\mathrm{d}x
- \int_{\mathbb{R}^d} \left( 1 - \psi\left(\frac{|x|}{R}\right) \right) \left( |\partial_t u_c (t)|^2
- |u_c (t)|^2 - \mu \frac{d}{d+2} |u_c (t)|^\frac{2(d+2)}d \right) \,\mathrm{d}x\\
& + \int_{\mathbb{R}^d} \frac{|x|}{2R} \psi'\left( \frac{|x|}{R}\right)
 \left( |\partial_t u_c (t)|^2 - |\nabla u_c (t)|^2 - |u_c (t)|^2 - \mu \frac{d}{d+2} |u_c (t)|^\frac{2(d+2)}d \right) \,\mathrm{d}x \\
& + \int_{\mathbb{R}^d} \frac1{ |x| R} \psi'\left( \frac{|x|}{R} \right) \left( x \cdot \nabla u_c (t)\right)^2 \, \mathrm{d}x.
\end{align*}
Then, by \eqref{eq2.13v178}, \eqref{eq3.1v161}, and \eqref{eq8.9}, we have for any $T_0 \le t \le T_1$,
\begin{align*}
Z_R'(t) \ge &\ \epsilon \left( \|u_c (t)\|_{H^1}^2 + \|\partial_t u_c \|_{L^2}^2 \right) - 2 \epsilon \eta_2 - 10 \eta_1 \\
& + \left(  ( 1 - 2 \epsilon) \left( 1 + \mu \frac{M(u_c (t))}{M(Q)} \right) - 2 \epsilon C(\eta_2)^2 \right) \int_{\mathbb{R}^d} |\nabla u_c (t)|^2 \,\mathrm{d}x,
\end{align*}
where
\begin{align}\label{eq3.41v182}
R = C(\eta_1) + \sup\limits_{t \in [T_0, T_1]} |x(t)|.
\end{align}

Taking $\eta_2$ sufficiently small depending on $u_c$, and $\epsilon$ small enough depending on $C(\eta_2)$, and finally $\eta_1$ sufficiently small depending on $\epsilon$ and $u_c $, we get
\begin{align}\label{eq8.11}
Z_R'(t) \gtrsim_{u_c} 1 , \quad \forall\,  T_0\le t \le T_1.  
\end{align}
By  \eqref{eq8.10}, \eqref{eq8.11}, \eqref{eq3.41v182}, and \eqref{eq8.6}, we obtain
\begin{align*}
T_1 - T_0 \lesssim_{u_c} C(\eta_1) + \eta_1 T_1, \quad \forall\,  
 T_1 > T_0.
\end{align*}
Taking $\eta_1$ sufficiently small depending on $u_c$ and $T_1$ large enough, we get a contradiction. Thus, we have proven Theorem \ref{th8.1}.

\section{Profile decomposition: proof of Theorem \ref{th5.1}}\label{se5}

In this section, we prove Theorem \ref{th5.1}. The main tools are the refined linear and bilinear Strichartz estimates. Let us introduce a set of weak limit modulo symmetry.
\begin{definition}
For a  bounded sequence ${\bf v}=\{v_n\}_n \subseteq H^1$, we let $\mathcal{V}({\bf v})$ be the set of all functions $\phi \in L^2$ such that
there exist a number $\Lambda >0$ and sequences $\{ \lambda_n \}_n \subseteq [\Lambda^{-1},\infty)$,
$\{ \xi_n \}_n \subseteq \Lambda [-1,1)^d$, $\{ t_n \}_n \subseteq \mathbb{R}$,
and $\{ x_n \}_n \subseteq \mathbb{R}^d$ such that
\[
	D_{\lambda_n}^{-1} \L_{\xi_n}^{-1} e^{i t_n \langle \nabla \rangle}  T_{x_n}^{-1} v_n \rightharpoonup \phi \quad
	\text{weakly in } L^2
\]
along a subsequence. Further, we let $\eta ({\bf v}):= \sup\limits_{\phi \in \mathcal{V}({\bf v})} \|\phi\|_{L^2}$.
\end{definition}
For a sequence ${\bf v}$ bounded in $H^1$, the case $\eta({\bf v})=0$ corresponds to the vanishing scenario.
We give a control of $\eta({\bf v})$, which is called the inverse Strichartz estimate.
As mentioned in Remark \ref{rem:improve}, the cube decomposition \eqref{eq2.11v165} is sufficient.
\begin{lemma}[Inverse Strichartz estimate]\label{L:ISE}
Let ${\bf w}=\{w_n\}_n$ be a bounded sequence in $H^1$. For any $M>0$ and $\varepsilon>0$, there exists $\alpha=\alpha(M,\varepsilon)>0$ such that
if
\[
	\| w_n \|_{H^1} \le M
\]
and
\[
	\limsup_{n\to\infty} \left\|e^{-it \langle \nabla \rangle} w_n \right\|_{L_{t,x}^\frac{2(d+2)}d} \ge \varepsilon,
\]
then
\[
	\eta ({\bf w}) \ge \alpha.
\]
\end{lemma}
\begin{proof}
Since $\left\| P_{>N} w_n \right\|_{H^{ \frac12}} \le MN^{- \frac12}$, we see from the Strichartz estimate \eqref{eq2.8v161} that there exists $N_0=N_0(M,\varepsilon) \in 2^\mathbb{Z}$ such that
\[
	\limsup_{n\to\infty} \left\|e^{-it \langle \nabla \rangle} P_{\le N_0} w_n \right\|_{L_{t,x}^\frac{2(d+2)}d} \ge \frac{\varepsilon}2.
\]
By \eqref{eq2.11v165}, there exists a dyadic cube $\mathcal{C}_n= \xi_n + \lambda_n^{-1} [-1,1)^d$ with $|\xi_n| \lesssim_{N_0} 1$
and $\lambda_n \gtrsim_{N_0} 1$
such that
\[
 	\lambda_n^{ \frac{d}2 }  \left\| P_{\mathcal{C}_n} e^{-i t \langle \nabla \rangle} w_n \right\|_{L^\I_{t,x}} \gtrsim_{M,\varepsilon} 1.
\]
Thus, one can choose $\left(t_n, x_n \right) \in \mathbb{R}\times \mathbb{R}^d$ so that
\begin{align}\label{eq4.14}
\left|\left(P_{\mathcal{C}_n} e^{i t_n \langle \nabla \rangle} w_n \right)  (-x_n) \right| \gtrsim_{M,\varepsilon}
\lambda_n^{-\frac{d}2}.
\end{align}
With the parameters $\xi_n, \lambda_n, x_n, t_n$ given above, we define
\begin{align}\label{eq4.15}
W_n := D_{\lambda_n}^{-1} \L_{\xi_n}^{-1} e^{i t_n \langle \nabla \rangle}  T_{x_n}^{-1}  w_n.
\end{align}
Since $\xi_n$ is uniformly bounded, $W_n$ is a bounded sequence in $L_x^2$.
Hence, after passing to a subsequence, there is a weak limit ${\phi} \in \mathcal{V}({\bf w})$.
Note that $\eta({\bf w}) \ge \|\phi\|_{L^2}$ by definition of $\eta$.

It suffices to show that there exists $\beta=\beta(M,\varepsilon)>0$ such that $\| \phi \|_{L^2}\ge \beta$.
To this end, we introduce
\begin{align*}
h_n := D_{\lambda_n}^{-1} \L_{\xi_n}^{-1} m_0^{\xi_n}(\nabla)^{-1} \hat{T}_{\xi_n} D_{\lambda_n} \mathcal{F}^{-1} {\bf 1}_{[-1,1)^d},
\end{align*}
where $\hat{T}_\xi = \mathcal{F} ^{-1} T_\xi \mathcal{F} = e^{i x\cdot \xi}$ is a multiplication operator.
In light of Lemma \ref{le2.8}, we see that $h_n$ converges to a function $h \in L^2$ strongly in $L^2$ along a subsequence.
Furthermore, one has $\|h\|_{L^2} \lesssim_{N_0(M,\varepsilon)} 1$.
We remark that
\begin{align*}
	\lambda_n^{\frac{d}2}\left(P_{Q_n} e^{-i t_n \langle \nabla \rangle}    w_n \right)  (-x_n)
	=  {}  &  (2\pi)^{-\frac{d}2}  \lambda_n^{\frac{d}2}  \int_{\mathbb{R}^d}  {\bf 1}_{\xi_n + \lambda_n^{-1}[-1,1)^d}(\xi)
 e^{- ix_n \cdot \xi}\mathcal{F} \left( e^{i t_n \langle \nabla \rangle}    w_n  \right)(\xi) d \xi \\
	={}&  (2\pi)^{-\frac{d}2} \int_{\mathbb{R}^d}  {\bf 1}_{[-1,1)^d}(z) \left(D_{\lambda_n}
T_{\xi_n}^{-1} \mathcal{F} \left( T_{x_n}^{-1} e^{i t_n \langle \nabla \rangle}    w_n \right) \right)(z) d z \\
={}&  (2\pi)^{-\frac{d}2} \left\langle D_{\lambda_n} T_{\xi_n}^{-1} \mathcal{F} \left( T_{x_n}^{-1}e^{i t_n \langle \nabla \rangle}  w_n \right) ,
 {\bf 1}_{[-1,1)^d}  \right\rangle_{L^2},
\end{align*}
where we have applied the change of variable $z=\lambda_n(\xi - \xi_n)$ to obtain the second line.
Plugging the identity $D_\lambda^{-1} = \mathcal{F}^{-1}  D_\lambda \mathcal{F}$ and using the unitarity of
$\mathcal{F}^{-1}$, $D_\lambda$, and $\hat{T}_\xi$ in $L^2$ and \eqref{eq3.14v162}, one sees that the right hand side equals to
\begin{align*}
	&(2\pi)^{-\frac{d}2} \left\langle  T_{x_n}^{-1}e^{i t_n \langle \nabla \rangle}  w_n ,
\hat{T}_{\xi_n} D_{\lambda_n}  \mathcal{F}^{-1} {\bf 1}_{[-1,1)^d}  \right\rangle_{L^2}\\
	= & (2\pi)^{-\frac{d}2} \left\langle \L_{\xi_n} D_{\lambda_n} W_n ,
\hat{T}_{\xi_n} D_{\lambda_n}  \mathcal{F}^{-1} {\bf 1}_{[-1,1)^d}  \right\rangle_{L^2}
 = (2\pi)^{-\frac{d}2} \left\langle W_n , h_n \right\rangle_{L^2}.
\end{align*}
Thus, by means of \eqref{eq4.14}, one has
\begin{align*}
\| \phi \|_{L^2} \gtrsim_{M,\varepsilon}  \left| \langle \phi  , h \rangle_{L^2} \right|
={}& \lim_{n\to\infty}  \left| \left\langle W_n, h_n  \right\rangle_{L^2} \right|\\
\ge {}& (2\pi)^{\frac{d}2}\liminf_{n\to\infty} \lambda_n^{\frac{d}2} \left| \left(P_{Q_n} e^{-i t_n \langle \nabla \rangle}  w_n \right)  (-x_n)\right|
\gtrsim_{M,\varepsilon} 1.
\end{align*}
Hence, the claim is proven.
\end{proof}
We next give two more characterization of the orthogonality of the parameters.
\begin{lemma}\label{L:orthchar}
Let $ \left\{ \left(\lambda_n,t_n,x_n, \nu_n \right) \right\}_{n\ge1}$ and $ \left\{ \left(\tilde\lambda_n,\tilde{t}_n,\tilde{x}_n, \tilde{\nu}_n \right) \right\}_{n\ge1}$
be two sequences of $\R_+ \times \R \times \R^d \times \R^d$ satisfying the normalization rule in Theorem \ref{th5.1}.
Then, the following three are equivalent:
\begin{enumerate}
\item \eqref{eq3.21v172} holds;
\item For any $\phi \in L^2$,
\[
	D_{{\tilde\lambda}_n}^{-1} \L_{{\tilde\nu}_n}^{-1} e^{-i \tilde{t}_n \langle \nabla \rangle}  T_{\tilde{x}_n}^{-1}
	\left(T_{x_n} e^{it_n \langle \nabla \rangle} \L_{\nu_n} D_{\lambda_n} P_n \phi \right) \rightharpoonup 0 \text{ in $L^2$, as $n\to\infty$,}
\]
 where $P_n$ is the projector defined as in \eqref{eq3.17v182} with $\lambda_n$;
\item For any subsequence $\{n_k\}_k$ of $\{n\}_n$, there exists a sequence of functions $\{ u_k \}_k$ bounded in $H^1$ such that, along a subsequence $\{k_\ell \}_\ell$,
\[
	D_{{\tilde\lambda}_{n_{k_\ell}}}^{-1} \L_{{\tilde\nu}_{n_{k_\ell}}}^{-1} e^{-i \tilde{t}_{n_{k_\ell}} \langle \nabla \rangle}  T_{\tilde{x}_{n_{k_\ell}}}^{-1}
	u_{k_\ell} \rightharpoonup 0, \quad
	D_{{\lambda}_{n_{k_\ell}}}^{-1} \L_{{\nu}_{n_{k_\ell}}}^{-1} e^{-i {t}_{n_{k_\ell}} \langle \nabla \rangle}  T_{{x}_{n_{k_\ell}}}^{-1}
	u_{k_\ell} \rightharpoonup u_\infty \neq 0,
\]
weakly in $L^2$ as $\ell \to \infty$.
\end{enumerate}
\end{lemma}
\begin{proof}
(1) $\Rightarrow$ (2).
We omit the subscript $n$ in this part since the role is less important.
Pick two functions $\phi,\psi \in L^2$. The goal is to show that for any $\varepsilon>0$, there exists $K = K(\varepsilon) \ge 1$ such that
\[	I :=  \left\langle D_{{\tilde\lambda}}^{-1} \L_{{\tilde\nu}}^{-1} e^{-i \tilde{t} \langle \nabla \rangle}  T_{\tilde{x}}^{-1}
	\left(T_{x} e^{it \langle \nabla \rangle} \L_{\nu} D_{\lambda} P_{\lambda} \phi \right),  \psi \right\rangle_{L^2}
\]
obeys the bound $|I| \le \varepsilon$ as long as
\[
	\frac{\lambda}{\tilde{\lambda}} + \frac{\tilde{\lambda}}{\lambda} +  \lambda |\nu - \tilde{\nu} | + \frac{|s_\Delta|}{\lambda^2} + \frac{|y_\Delta|} {\lambda} \ge K,
\]
where $P_\lambda$ is a suitable frequency cutoff,
$ \left(-s_\Delta,y_\Delta \right)=L_{\tilde{\nu}} \left(\tilde{t}-t,\tilde{x}-x \right)$, and $\nu, \tilde{\nu} \in B(0,\Lambda)$ for some $\Lambda$.
By the density argument, we may suppose without loss of generality that
$\phi, \psi \in \mathcal{S}$ and $\supp \mathcal{F} \phi, \supp \mathcal{F} \psi \subseteq \overline{ B(0,R_0)} $ for some $R_0 \gg 1$. Furthermore, by this modification, we may replace $P_\lambda$ by the identity.

Let us begin with the proof when $\frac{\lambda}{\tilde{\lambda}} + \frac{\tilde{\lambda}}{\lambda}$ is large.
Note that
if  the Fourier support of a function $f$ is a subset of $B(c,R)$, then those of $D_\lambda f$ and $\L_\nu f$ are included
in $B \left(\frac{c}\lambda , \frac{R}\lambda \right)$ and $B\left(  l_{\nu}(c), 2\langle \nu \rangle R \right)$, respectively.
Hence
\begin{equation}\label{E:orthcharpf12-1}
	\supp \mathcal{F}  \left(D_{{\tilde\lambda}}^{-1} \L_{{\tilde\nu}}^{-1} e^{-i \tilde{t} \langle \nabla \rangle}  T_{\tilde{x}}^{-1}
	\left(T_{x} e^{it \langle \nabla \rangle} \L_{\nu} D_{\lambda} \phi \right) \right)
\subseteq B \left(  \tilde{\lambda} l_{-\tilde{\nu}} \left(l_{\nu} (0) \right) , 4 \langle \tilde{\nu} \rangle \langle \nu \rangle R_0 \tfrac{\tilde\lambda}{\lambda}  \right).
\end{equation}
By Bernstein's inequality and boundedness of $\nu_n$ and $\tilde{\nu}_n$,
\begin{align*}
	|I|
	&{}\le \normo{D_{{\tilde\lambda}}^{-1} \L_{{\tilde\nu}}^{-1} e^{-i \tilde{t} \langle \nabla \rangle}  T_{\tilde{x}}^{-1}
	\left(T_{x} e^{it \langle \nabla \rangle} \L_{\nu} D_{\lambda} \phi \right) }_{L_x^{\frac{2(d+2)}d} } \norm{\psi}_{L_x^{\frac{2(d+2)}{d+4}}}  \\
	&{}\lesssim_\psi
	\left(4 \langle \tilde{\nu} \rangle \langle \nu \rangle R_0 \tfrac{\tilde{\lambda}}{\lambda} \right)^{\frac{d}{d+2}}
	\normo{D_{{\tilde\lambda}}^{-1} \L_{{\tilde\nu}}^{-1} e^{-i \tilde{t} \langle \nabla \rangle}  T_{\tilde{x}}^{-1}
	\left(T_{x} e^{it \langle \nabla \rangle} \L_{\nu} D_{\lambda} \phi \right)}_{L_x^2}   \lesssim_{\Lambda,R_0}
	\left(\tfrac{\tilde{\lambda}}{\lambda}  \right)^{\frac{d}{d+2}} \norm{\phi}_{L^2} \to 0
\end{align*}
as $\frac\lambda{\tilde\lambda}\to\infty$.

Hence, there exists $K_1 = K_1(\varepsilon)$ such that we obtain the desired smallness if $\frac\lambda{\tilde\lambda} \ge K_1$.
On the other hand, in the case when $\frac{\tilde{\lambda}}{\lambda}$ is large, we use the identity
\begin{align*}
I=
\left\langle \phi, D_{{\lambda}}^{-1} \L_{{\nu}}^{-1} \left(m_0^{\nu}(\nabla)^{-1} m_0^{\tilde\nu}(\nabla) \right) e^{-i {t} \langle \nabla \rangle}  T_{{x}}^{-1}
\left(T_{\tilde{x}} e^{i\tilde{t} \langle \nabla \rangle} \L_{\tilde\nu} D_{\tilde\lambda} \psi \right) \right\rangle_{L^2}.
\end{align*}
Since $m_0^{\nu}(\nabla)^{-1} m_0^{\tilde\nu}(\nabla)$ is a multiplication by a bounded factor in the Fourier side,
we obtain
\[
\normo{D_{{\lambda}}^{-1} \L_{{\nu}}^{-1}  \left(m_0^{\nu}(\nabla)^{-1} m_0^{\tilde\nu}(\nabla) \right) e^{-i {t} \langle \nabla \rangle}  T_{{x}}^{-1}
\left(T_{\tilde{x}} e^{i\tilde{t} \langle \nabla \rangle} \L_{\tilde\nu} D_{\tilde\lambda} \psi \right)}_{L_x^{\frac{2(d+2)}d}}
\lesssim  \left(\tfrac{\lambda}{\tilde{\lambda}} \right)^{\frac{d}{d+2}} \norm{\psi}_{L_x^2} \to 0
\]
as $\frac{\tilde\lambda}\lambda\to\infty$, just as in the previous case.
Thus, replacing $K_1$ with a larger one if necessary, we obtain the desired smallness if $\frac{\tilde\lambda}\lambda \ge K_1$.

We suppose $\frac{\lambda}{\tilde{\lambda}} + \frac{\tilde{\lambda}}{\lambda}\le 2K_1$ in what follows.
Let us next consider the case $\lambda= \tilde\lambda = 1$.
In this case, we have $\nu = \tilde{\nu} = 0$ by the normalization condition.
Furthermore, $ \left(-s_\Delta,y_\Delta \right)= \left(\tilde{t}-t,\tilde{x}-x \right)$. One has
\[
I= \left\langle T_{x-\tilde{x}} e^{i \left(t- \tilde{t} \right) \langle \nabla \rangle} \phi, \psi  \right\rangle_{L^2}.
\]
By a standard argument, one sees that there exists $K_2 = K_2(\varepsilon)$ such that the modulus of
the right hand side is smaller than $\varepsilon$ if $ \left|t-\tilde{t} \right| +  \left|x-\tilde{x} \right|\ge K_2$.

Let us move to the case $\lambda \to \infty$. Note that $\tilde\lambda \ge \frac\lambda{2K_1} \to \infty$ by our assumption.
We next consider the case when $\lambda \left|\nu-\tilde{\nu} \right|$ is sufficiently large.
A computation shows
\[
	\tilde{\lambda} \left| l_{-\tilde{\nu}}  \left(l_{\nu} (0) \right) \right|
= \tilde{\lambda} \sqrt{  \left(\langle \nu \rangle \langle \tilde{\nu}\rangle
- \nu \cdot \tilde{\nu} + 1 \right) \left(\langle \nu \rangle \langle \tilde{\nu} \rangle - \nu \cdot \tilde{\nu} - 1 \right) }
	\ge \sqrt{\frac{\min  \left( \langle \nu \rangle , \langle \tilde{\nu} \rangle \right)}{\max  \left( \langle \nu \rangle , \langle \tilde{\nu} \rangle \right)}} \frac{\lambda  \left|\nu-\tilde\nu \right|}{2K_1} .
\]
Hence, there exists $K_3$ such that if $\lambda \left|\nu-\tilde{\nu} \right| \ge K_3$ then $I=0$ is deduced from the disagreement of
the Fourier support. We therefore suppose that $\lambda \left|\nu-\tilde{\nu} \right| \le K_3$. One sees from \eqref{eq3.12v161} that
\[
\L_{{\tilde\nu}}^{-1} e^{-i \tilde{t} \langle \nabla \rangle}  T_{\tilde{x}}^{-1} 	T_{x} e^{it \langle \nabla \rangle}
	= \L_{{\tilde\nu}}^{-1} T_{x-\tilde{x}} e^{i \left(t-\tilde{t} \right) \langle \nabla \rangle}
	= T_{-y_\Delta} e^{i s_\Delta \langle \nabla \rangle} \L_{{\tilde\nu}}^{-1}.
\]
Hence,
\[
	I = \left\langle T_{-{y_\Delta}/{\tilde\lambda }} e^{i  \left(s_\Delta/\tilde\lambda^2 \right) \langle \nabla \rangle}  D_{\tilde{\lambda}}^{-1}  \L_{{\tilde\nu}}^{-1} \L_\nu D_\lambda \phi ,  \psi  \right\rangle_{L^2}.
\]
One verifies that $D_{\tilde{\lambda}}^{-1}  \L_{{\tilde\nu}}^{-1} \L_\nu D_\lambda \phi$ takes value in a precompact set in $L^2$. Hence, there exists $K_4$ such that if $|s_\Delta|/\tilde\lambda^2 + |y_\Delta|/\tilde\lambda \ge K_4$, then
$|I| \le \varepsilon$.

Combining the above together, we prove the existence of the desired $K$.
\medskip

(2) $\Rightarrow$ (3).
Pick $u_\infty \in H^1$, $u_\infty \neq0$, and set $u_k := T_{x_{n_k}} e^{it_{n_k} \langle \nabla \rangle} \L_{\nu_{n_k}} D_{\lambda_{n_k}} P_{n_k} u_\infty$.
\medskip

(3) $\Rightarrow$ (1).
We prove by contradiction. Suppose that (1) fails. Then, there exists a subsequence $\{n_k\}_{k\ge1}$
such that
\[
	\nu_{n_k} \to \nu_\infty\in \R^d, \text{  as }  \tilde{\nu}_{n_k} \to \tilde\nu_\infty
\]
and
\[
	\frac{\lambda_{n_k}}{\tilde\lambda_{n_k}} \to \lambda_* \in (0,\infty), \
	\tilde\lambda_{n_k} \left|\nu_{n_k} - \tilde\nu_{n_k}  \right| \to \nu_*\in [0,\infty),  \
	\frac{s_{\Delta,n_k}}{\lambda_{n_k}^2} \to s_{*}\in \R , \
	\frac{y_{\Delta,n_k}}{\lambda_{n_k} } \to y_{*}\in \R^d, \text{ as $k\to\infty$,}
\]
 where
$ \left(-s_{\Delta,n},y_{\Delta,n} \right)=L_{\tilde{\nu}_n} \left(\tilde{t}_n-t_n,\tilde{x}_n-x_n \right)$.
Along this sequence, the operator
\[
	S_n:= D_{{\tilde\lambda}_n}^{-1} \L_{{\tilde\nu}_n}^{-1} e^{-i \tilde{t}_n \langle \nabla \rangle}  T_{\tilde{x}_n}^{-1}
	T_{x_n} e^{it_n \langle \nabla \rangle} \L_{\nu_n} D_{\lambda_n}
=  T_{-y_{\Delta,n}/{\tilde\lambda}_n} e^{i  \left(s_{\Delta,n}/{\tilde\lambda}_n^2 \right) \langle \nabla \rangle} \L_{{\tilde\nu_n}}^{-1}	
\]
converges to a bounded operator, say $S_\infty$, in the strong operator sense.

Now, we suppose that a bounded sequence $ \left\{u_k \right\}_{k\ge1} \subseteq H^1$ satisfies
\[
	D_{{\tilde\lambda}_{n_{k_\ell}}}^{-1} \L_{{\tilde\nu}_{n_{k_\ell}}}^{-1} e^{-i \tilde{t}_{n_{k_\ell}} \langle \nabla \rangle}  T_{\tilde{x}_{n_{k_\ell}}}^{-1}
	u_{k_\ell} \rightharpoonup 0, \
	D_{{\lambda}_{n_{k_\ell}}}^{-1} \L_{{\nu}_{n_{k_\ell}}}^{-1} e^{-i {t}_{n_{k_\ell}} \langle \nabla \rangle}  T_{{x}_{n_{k_\ell}}}^{-1}
	u_{k_\ell} \rightharpoonup u_\infty, \text{ in } L^2, \text{ as } \ell\to\infty.
\]
Since
\[
D_{{\tilde\lambda}_{n_{k_\ell}}}^{-1} \L_{{\tilde\nu}_{n_{k_\ell}}}^{-1} e^{-i \tilde{t}_{n_{k_\ell}} \langle \nabla \rangle}
 T_{\tilde{x}_{n_{k_\ell}}}^{-1} 	u_{k_\ell}
=S_n D_{{\lambda}_{n_{k_\ell}}}^{-1} \L_{{\nu}_{n_{k_\ell}}}^{-1} e^{-i {t}_{n_{k_\ell}} \langle \nabla \rangle}  T_{{x}_{n_{k_\ell}}}^{-1}
	u_{k_\ell},
\]
we see from the uniqueness of the weak limit that $S_\infty^* u_\infty =0$, where $S_\infty^*$ is an adjoint operator of $S_\infty$. This implies $u_\infty$ must be zero. Hence (3) fails. We complete the proof.
\end{proof}

Now, we are ready to prove Theorem \ref{th5.1}.
\begin{proof}[Proof of Theorem \ref{th5.1}]
Assume ${\bf w}^0 = \left\{v_n\right\}_{n\ge 1}$ is a bounded sequence in $H_x^1(\mathbb{R}^d)$ satisfying $\norm{v_n}_{H^1}\leq A$.
We divide the proof into five steps. In the first three steps, we construct profiles, parameters, and remainders by induction.
The fourth step is devoted to the mutual orthogonality of the parameters.
In the last step we establish the smallness of the remainders. We remark that we freely extract a subsequence of $n$, denoted again by $n$.

{\bf Step 1. Construction of the first profile and the first remainder.}
If $\eta \left({\bf w}^0 \right)=0$, then one has the desired property with the choice $J_0=1$.
Hence, we suppose $\eta \left({\bf w}^0 \right)>0$ in the sequel.

By the definition of $\eta$, there exists $\tilde{\phi}^1 \in \mathcal{V} \left({\bf w}^0 \right) $ such that
$\left\|\tilde{\phi}^1 \right\|_{L^2} \ge \frac12 \eta \left({\bf w}^0 \right)>0$.
By the definition of $\mathcal{V} \left({\bf w}^0 \right)$,
there exists $ \left(\tilde{\lambda}_n^1,\tilde{\xi}_n^1,\tilde{t}_n^1, \tilde{x}_n^1 \right)\in \R_+ \times \R^d \times \R \times \R^d$ such that
\begin{equation}\label{E:w1weaklimit}
D_{\tilde{\lambda}_n^1}^{-1} \L_{\tilde{\xi}_n^1}^{-1} e^{i \tilde{t}_n^1 \langle \nabla \rangle}  T_{\tilde{x}_n^1}^{-1} v_n \rightharpoonup \tilde{\phi}^1 \quad
\text{ weakly in } L^2
\end{equation}
along a subsequence.
Furthermore, $\tilde{\lambda}_n^1$ and $ \left|\tilde{\xi}_n^1 \right|$ are bounded by a positive constant from below and above, respectively.
Further extracting a subsequence if necessary, we have
\[
\tilde{\lambda}_n^1 \to \tilde{\lambda}_\infty^1 \in (0,\infty],\quad
\tilde{\xi}_n^1 \to \tilde{\xi}_\infty^1 \in \mathbb{R}^d, \quad
\left(\tilde{\lambda}_n^1 \right)^{-2}	\tilde{t}_n^1   \to \tilde{\tau}_\infty^1 \in [-\infty,\infty],
\quad
e^{i \tilde{t}_n^1/ \left\langle \xi_n^1  \right\rangle} \to e^{i\theta^1} \in \mathbb{C}.
\]
Now we modify the profile and parameters so that the parameter satisfies the desired property, i.e. either $\lim\limits_{n\to\infty}\lambda_n^1= \infty$ or $\lambda_n^1\equiv 1$, etc.
\begin{itemize}
\item
If $\tilde{\lambda}^1_\infty < \infty$ and $\tilde{\tau}_\infty^1\in \R$, then we take
\[
	\phi^1:=e^{i \tilde{\tau}_\infty^1  \left(\lambda_\infty^1 \right)^2 \langle \nabla \rangle}  \L_{\xi_\infty^1} D_{\lambda_\infty^1}  \tilde{\phi}^1, \quad  \lambda_n^1:=1, \quad \nu_n^1:=0, \quad t_n^1:=0, \quad x_n^1:= \tilde{x}_n^1.
\]
Note that $\tilde{t}_n^1 \to \tilde{\tau}_\infty^1 (\lambda_\infty^1)^2 \in \mathbb{R}$ as $n\to\infty$.
\item If $\tilde{\lambda}^1_\infty < \infty$ and $\tilde{\tau}_\infty^1  = \pm \infty $, then we take
\[
 \phi^1:=  \L_{\xi_\infty^1} D_{\lambda_\infty^1}  \tilde{\phi}^1, \quad  \lambda_n^1:=1, \quad \nu_n^1:=0, \quad t_n^1:=\tilde{t}_n^1,
 \quad x_n^1:= \tilde{x}_n^1.
\]
\item If $\tilde{\lambda}^1_\infty=\infty$ and $\tilde{\tau}_\infty^1\in \R$, then we take
\[
 \phi^1:=e^{-i\theta^1} e^{-i \frac{\tilde{\tau}_\infty^1 \Delta }{2 \left\langle \tilde{\xi}_\infty^1  \right\rangle}}  \tilde{\phi}^1, \quad  \lambda_n^1:=\tilde{\lambda}_n^1, \quad \nu_n^1:=   \tilde{\xi}_n^1, \quad t_n^1:=0, \quad x_n^1:=  \tilde{x}_n^1+ \frac{\tilde{\xi}_n^1}{ \left\langle \tilde{\xi}_n^1  \right\rangle} \tilde{t}_n^1.
\]
\item If $\tilde{\lambda}^1_\infty=\infty$ and
 $\tilde{\tau}_\infty^1= \pm \infty$, then we simply take
\[
 \phi^1:=\tilde{\phi}^1, \quad  \lambda_n^1:=\tilde{\lambda}_n^1, \quad \nu_n^1:=\tilde{\xi}_n^1, \quad t_n^1:=\tilde{t}_n^1, \quad x_n^1:= \tilde{x}_n^1.
\]
\end{itemize}
Note that one has
$\lim\limits_{n\to\infty} t_n^1= \tilde{\tau}_\infty^1 \in \{ \pm \infty\}$ in the second and the fourth cases.
Set $\lambda_\infty^1:= \lim\limits_{n\to\infty} \lambda_n^1 \in \{1,\infty\}$.
It follows from \eqref{eq3.14v162} that
$ \left\| \phi^1  \right\|_{L^2} \gtrsim  \left\langle \xi_\infty^1 \right\rangle^{-1}  \left\|\tilde{\phi}^1 \right\|_{L^2}$.
In particular, $\phi^1 \neq0$.
Let us now prove that
\begin{equation}\label{E:1stweaklimit}
	W_{n}^{0,1}:= D_{{\lambda}_n^1}^{-1} \L_{{\nu}_n^1}^{-1} e^{i {t}_n^1 \langle \nabla \rangle}  T_{x_n^1}^{-1} v_n \rightharpoonup {\phi}^1 \quad
	\text{weakly in } L^2
\end{equation}
along the same subsequence. In the last case, this is nothing but \eqref{E:w1weaklimit}.
Furthermore, one easily verifies it is also in the first two cases by the convergence of the parameters.
Let us consider the third case. By \eqref{eq3.12v161}, we have
\[
 \L_{\tilde{\xi}_n^1}^{-1} e^{i \tilde{t}_n^1 \langle \nabla \rangle} T_{\frac{\tilde{\xi}_n^1}{ \left\langle \tilde{\xi}_n^1 \right\rangle} \tilde{t}_n^1}
 = e^{i \frac{\tilde{t}_n^1}{ \left\langle \tilde{\xi}_n^1  \right\rangle} \langle \nabla \rangle} \L_{\tilde{\xi}_n^1}^{-1},
\]
from which we obtain
\[
	D_{\tilde{\lambda}_n^1}^{-1} \L_{\tilde{\xi}_n^1}^{-1} e^{i \tilde{t}_n^1 \langle \nabla \rangle}  T_{\tilde{x}_n^1}^{-1} v_n
= e^{i \frac{\tilde{t}_n^1}{\langle \xi_n \rangle}  \left\langle  \left(\lambda_n^1 \right)^{-1} \nabla \right\rangle}
\left( D_{{\lambda}_n^1}^{-1} \L_{\nu_n^1}^{-1} e^{i {t}_n^1 \langle \nabla \rangle} T_{x_n^1}^{-1} v_n \right).
\]
Then, one can extract a subsequence so that the strong operator convergence
\[
	e^{i \frac{\tilde{t}_n^1}{ \left\langle \tilde{\xi}_n^1  \right\rangle}  \left\langle  \left(\lambda_n^1 \right)^{-1} \nabla  \right\rangle} =
	e^{i \frac{\tilde{t}_n^1}{ \left\langle \tilde{\xi}_n^1  \right\rangle}  \left( \left\langle  \left(\lambda_n^1 \right)^{-1} \nabla  \right\rangle-1 \right)}
	e^{i \frac{\tilde{t}_n^1}{ \left\langle \tilde{\xi}_n^1  \right\rangle}}
	\to e^{i \frac{\tilde{\tau}_\infty^1 }{2 \left\langle \tilde{\xi}_\infty^1  \right\rangle}\Delta } e^{i \theta^1}
\]
holds as $n\to\infty$ as an operator from $L^2$ into itself. Since the limit operator is unitary in $L^2$, we obtain
\[
	D_{{\lambda}_n^1}^{-1} \L_{\nu_n^1}^{-1} e^{i {t}_n^1 \langle \nabla \rangle} T_{x_n^1}^{-1} v_n
	\rightharpoonup  \Big(e^{i \frac{\tilde{\tau}_\infty^1 }{2 \left\langle \tilde{\xi}_\infty^1 \right\rangle} \Delta } e^{i \theta^1}  \Big)^{-1} \tilde{\phi}^1 = \phi^1
\]
as desired. We obtain \eqref{E:1stweaklimit}.

Furthermore, if $\lambda_n^1\equiv0$, that is $\tilde{\lambda}_\infty^1<\infty$,
one sees from the boundedness of $ \left\{\nu_n^1 \right\}_n$ that
the sequence $ \left\{ D_{{\lambda}_n^1}^{-1} \L_{{\nu}_n^1}^{-1} e^{i {t}_n^1 \langle \nabla \rangle}  T_{x_n^1}^{-1} v_n \right\}_{n\ge1}$ is bounded in $H^1$. Hence, one has $\phi^1 \in H^1$ from \eqref{E:1stweaklimit}.
Furthermore, the same weak convergence as in \eqref{E:1stweaklimit} holds in the weak $H^1$ sense.

Now, we define the remainder term ${\bf w}^1 = \{w_n^1\}_{n\ge1}$ by
\[
w_n^1 := w_n^0 - T_{x_n^1} e^{- i {t}_n^1 \langle \nabla \rangle} \L_{{\nu}_n^1} D_{{\lambda}_n^1}P_n^1 \phi^1.
\]
Then, the decomposition for $J=1$ immediately follows.
Moreover, by the virtue of the presence of $P_n^1$, $ \left\{w_n^1 \right\}$ is bounded in $H^1$.
Furthermore, as $P_n^1$ converges to the identity in the strong operator sense,
we have
\begin{equation}\label{E:2ndweaklimit1}
W_{n}^{1,1}:= D_{{\lambda}_n^1}^{-1} \L_{{\nu}_n^1}^{-1} e^{i {t}_n^1 \langle \nabla \rangle}  T_{x_n^1}^{-1} w_n^1 \rightharpoonup 0
	\text{ in } L^2, \text{ as } n \to \infty.
\end{equation}
Note that if $\lambda_n^1\equiv0$ then the convergence holds weakly in $H^1$ as in \eqref{E:1stweaklimit}.

 {\bf Step 2. Proof of the decoupling identities.}
Let us claim
\EQ{\label{E:decouplingpf1}
\normo{w_n^0}_{L^2}^2={}&\normo{w_n^{1}}_{L^2}^2+\normo{T_{x^1_n} e^{-i t^1_n \langle \nabla \rangle} \L_{\xi^1_n}  D_{\lambda^1_n} P_{n}^1\phi^1}_{L^2}^2+o(1),\\
\normo{w_n^0}_{\dot H^1}^2={}&\normo{w_n^{1}}_{\dot H^1}^2+\normo{T_{x^1_n} e^{-i t^1_n \langle \nabla \rangle} \L_{\xi^1_n}  D_{\lambda^1_n} P_{n}^1 \phi^1}_{\dot H^1}^2+o(1), \text{ as } n \to \infty.
}
For $k=0,1$, we have
\begin{align*}
	\left\langle \nabla^k w_n^0, \nabla^k w_n^0 \right\rangle_{L^2}
	={}&  \left\|  w_n^1 \right\|_{\dot{H}^k}^2 +  \left\| T_{x_n^1} e^{- i {t}_n^1 \langle \nabla \rangle}  \L_{{\nu}_n^1}
	D_{{\lambda}_n^1}P_n^1 \phi^1  \right\|_{\dot{H}^k}^2 \\
	& +  2 \Re  \left\langle \nabla^k w_n^1, \nabla^k T_{x_n^1} e^{- i {t}_n^1 \langle \nabla \rangle} \L_{{\nu}_n^1}
	D_{{\lambda}_n^1}P_n^1 \phi^1   \right\rangle_{L^2}.
\end{align*}
Hence, it suffices to show that the last term of the right hand tends to zero as $n\to\infty$.
If $\lambda_n^1\equiv0$, then it is a direct consequence of $\nu_n^1 \equiv 0$ and the fact that the weak convergence \eqref{E:1stweaklimit} holds weakly in $H^1$. We consider the case $\lambda_n^1\to\infty$ as $n\to\infty$, one has
\begin{align*}
\left\langle w_n^1 , T_{x_n^1} e^{- i {t}_n^1 \langle \nabla \rangle} \L_{{\nu}_n^1}
	D_{{\lambda}_n^1}P_n^1 \phi^1  \right\rangle_{{H}_x^k}
& = \left\langle W_n^{1,1} , D_{\lambda_n}^{-1} \langle \nabla \rangle^{2k} m_k^{\nu_n^1}(\nabla)^{-1} D_{\lambda_n^1} P_{\le (\lambda_n^1)^\theta} \phi \right\rangle_{L^2}.
\end{align*}
Since
\[
D_{\lambda_n}^{-1}  \langle \nabla \rangle^{2k} m_k^{\nu_n^1}(\nabla)^{-1} D_{\lambda_n^1} P_{\le \left(\lambda_n^1 \right)^\theta} \to \langle \nu_\infty^1 \rangle^{1-2k}
\]
in the strong operator sense in $\mathcal{L} \left(L^2 \right)$, we see from \eqref{E:2ndweaklimit1} that
\[
	\Re  \left\langle \nabla^k w_n^1, \nabla^k T_{x_n^1} e^{- i {t}_n^1 \langle \nabla \rangle} \L_{{\nu}_n^1}
	D_{{\lambda}_n^1}P_n^1 \phi^1    \right\rangle_{L^2} \to 0, \text{ as $n\to\infty$.}
\]
Therefore, we obtain \eqref{E:decouplingpf1}. It also yields
\EQ{
	\limsup_{n\to\infty} \normo{w_n^1}_{H^1} \le  \limsup_{n\to\infty} \normo{w_n^0}_{H^1} = A.
}
Furthermore, mimicking the proof of the claim one also has
\begin{equation}\label{E:PDsmallnesspf1}
	\normo{w_n^0}_{H^{ \frac12}}^2=  \normo{w_n^{1}}_{H^{ \frac12}}^2+\normo{T_{x^1_n} e^{-i t^1_n \langle \nabla \rangle} \L_{\xi^1_n}  D_{\lambda^1_n} P_{n}^1 \phi^1}_{H^{ \frac12}}^2 + o(1), \text{ as $n\to\infty$.}
\end{equation}
We now turn to the energy decoupling for $J=1$.
By \eqref{E:decouplingpf1}, it is enough to prove
\begin{align}\label{E:energydcpf1}
 \left\|\Re w_n^0  \right\|_{L_x^\frac{2(d+2)}d}^\frac{2(d+2)}d -  \left\|\Re w_n^1  \right\|_{L_x^\frac{2(d+2)}d}^\frac{2(d+2)}d
 -  \left\|\Re  \left(T_{x^1_n} e^{-i t^1_n \langle \nabla \rangle} \L_{\xi^1_n}  D_{\lambda^1_n} P_{n}^1 \phi^1 \right) \right\|_{L_x^\frac{2(d+2)}d}^\frac{2(d+2)}d \to 0,
  \text{ as $n\to\infty$.}
\end{align}
When $\lambda_n^1 \equiv 1$, $\nu_n^1 \equiv 0$, $t_n^1 \equiv 0$, we see
from  \eqref{E:1stweaklimit} and \eqref{E:2ndweaklimit1} that
\begin{align*}
	& \left\|\Re w_n^0  \right\|_{L_x^\frac{2(d+2)}d} =
	\left\|\Re T_{x_n^1}^{-1}w_n^0  \right\|_{L_x^\frac{2(d+2)}d}
\to
	\left\|\Re \phi^1  \right\|_{L_x^\frac{2(d+2)}d}
\intertext{ and }
& \left\|\Re w_n^1 \right\|_{L_x^\frac{2(d+2)}d} =
	\left\|\Re T_{x_n^1}^{-1}w_n^1  \right\|_{L_x^\frac{2(d+2)}d} \to 0, \text{ as $n\to\infty$.}
\end{align*}
Together with
\[
	\left\|\Re  \left(T_{x^1_n} \phi^1 \right)  \right\|_{L_x^\frac{2(d+2)}d} =
	 \left\|\Re \phi^1 \right\|_{L_x^\frac{2(d+2)}d},
\]
we have \eqref{E:energydcpf1}.

When $\lambda_n^1 \equiv 1$, $\nu_n^1 \equiv 0$, $t_n^1 \to t_\infty^1 \in \{\pm \infty \}$, we see
from the dispersive estimate that
\[
\left\|\Re  \left(T_{x^1_n} e^{-i t^1_n \langle \nabla \rangle} \phi^1 \right)  \right\|_{L_x^\frac{2(d+2)}d}
\le  \left\|e^{-i t^1_n \langle \nabla \rangle} \phi^1  \right\|_{L_x^\frac{2(d+2)}d}
\to 0, \text{ as $n\to\infty$.}
\]
 Hence, together with the embedding $ H^1  \hookrightarrow  L^{ \frac{2(d+2)}d}  $ and
the uniform boundedness of $ \left\{w_n^0 \right\}_{n\ge1}$ and $ \left\{w_n^1 \right\}_{n \ge 1}$, one has
\begin{align*}
& \left\lvert
	\left\|\Re w_n^0 \right\|_{L_x^\frac{2(d+2)}d}^\frac{2(d+2)}d -  \left\|\Re w_n^1  \right\|_{L_x^\frac{2(d+2)}d}^\frac{2(d+2)}d -  \left\|\Re  \left(T_{x^1_n} e^{-i t^1_n \langle \nabla \rangle} \phi^1 \right)  \right\|_{L_x^\frac{2(d+2)}d}^\frac{2(d+2)}d
	\right\rvert\\
& \lesssim_d  \left(  \left\|w_n^0 \right\|_{L_x^\frac{2(d+2)}d} +  \left\|w_n^1 \right\|_{L_x^\frac{2(d+2)}d} \right)^{\frac{d+4}{d}}
	\left\|\Re  \left(T_{x^1_n} e^{-i t^1_n \langle \nabla \rangle} \phi^1 \right)  \right\|_{L_x^\frac{2(d+2)}d}
	\to 0, \text{ as $n\to\infty$.}
\end{align*}
When $\lambda_\infty^1 = \infty$, by Bernstein's inequality, \eqref{eq3.14v162},
 $\sup\limits_\xi \left|\partial_{\xi_j} l_{\nu_n^1}(\xi)\right| \lesssim  \left\langle \nu_n^1  \right\rangle$,
 and the boundedness of $\nu_n^1$, we have
\begin{align*}
&  \left\| e^{-i t^1_n \langle \nabla \rangle} \L_{\xi^1_n}  D_{\lambda^1_n} P_{n}^1 \phi^1  \right\|_{L_x^\frac{2(d+2)}d}\\
&  \lesssim \left( \diam \left(\supp \mathcal{F}\left( e^{-i t^1_n \langle \nabla \rangle} \L_{\xi^1_n}  D_{\lambda^1_n} P_{n}^1 \phi^1 \right) \right) \right)^{\frac{d}{d+2}} \left\|e^{-i t^1_n \langle \nabla \rangle} \L_{\xi^1_n}  D_{\lambda^1_n} P_{n}^1 \phi^1 \right\|_{L_x^2}\\
& =  \left( \diam \left( \supp  \mathcal{F}\left( \L_{\xi^1_n}  D_{\lambda^1_n} P_{n}^1 \phi^1 \right)\right) \right)^\frac{d}{d+2} \left\| \L_{\nu_n^1} D_{\lambda_n^1} P_n^1\phi^1 \right\|_{L_x^2} \\
&  \lesssim \left(   \left\langle \nu_n^1  \right\rangle \diam\left( \supp \left(\mathcal{F}
\left( D_{\lambda_n^1} P_n^1 \phi^1\right)\right)\right)\right)^\frac{d}{d+2} \langle \nu_n^1 \rangle  \left\|\phi^1 \right\|_{L_x^2}
 \lesssim_{\sup |\nu_n^1|}  \left(\lambda_n^1 \right)^{\frac{d(\theta - 1)}{d+2}}  \left\| \phi^1  \right\|_{L_x^2} \to 0
\end{align*}
as $n\to\infty$. Then, we obtain \eqref{E:energydcpf1} as in the previous case.

{\bf Step 3. Construction of profiles and remainders by induction.}
Let us construct the other profiles and the remainders by induction.

Suppose that $\eta({\bf w}^k)>0$ for some $k\ge1$.
By the definition of $\eta$, there exists $\tilde{\phi}^{k+1} \in \mathcal{V} \left({\bf w}^k \right) $ such that
$ \left\|\tilde{\phi}^{k+1} \right\|_{L^2} \ge \frac12 \eta \left({\bf w}^k \right)>0$.
By definition of $\mathcal{V} \left({\bf w}^k \right)$, there exists $ \left(\tilde{\lambda}_n^{k+1},\tilde{\xi}_n^{k+1},\tilde{t}_n^{k+1}, \tilde{x}_n^{k+1} \right)\in \R_+ \times \R^d \times \R \times \R^d$ such that
\begin{equation}\label{E:wkweaklimit_tmp}
	D_{\tilde{\lambda}_n^{k+1}}^{-1} \L_{\tilde{\xi}_n^{k+1}}^{-1} e^{i \tilde{t}_n^{k+1} \langle \nabla \rangle}  T_{\tilde{x}_n^{k+1}}^{-1} w_n^k \rightharpoonup \tilde{\phi}^{k+1}
	\text{ in } L^2
\end{equation}
along a subsequence. Furthermore, $\tilde{\lambda}_n^{k+1}$ and $\big|\tilde{\xi}_n^{k+1} \big|$ are bounded by a positive constant from below and above, respectively.
Mimicking the argument in Step 1, one obtains $\phi^{k+1} \in L^2$, $\phi^{k+1}\neq0$,
and the parameter $ \left(\lambda_n^{k+1},\nu_n^{k+1},t_n^{k+1},x_n^{k+1} \right)$ satisfying the property of the theorem such that
\begin{equation}\label{E:kthweaklimit1}
	W_{n}^{k,{k+1}}:= D_{{\lambda}_n^{k+1}}^{-1} \L_{{\nu}_n^{k+1}}^{-1} e^{i {t}_n^{k+1} \langle \nabla \rangle}  T_{x_n^{k+1}}^{-1} w_n^k \rightharpoonup {\phi}^{k+1}
	\text{  in } L^2
\end{equation}
along a subsequence.
If $\lambda_n^{k+1} \equiv 1$, then $\phi^{k+1} \in H^1$ and
the weak convergence \eqref{E:kthweaklimit1} holds weakly in $H^1$.
By using the parameter, we define the remainder term
${\bf w}^{k+1} = \left\{w_n^{k+1} \right\}_{n\ge1}$ by
\begin{equation}\label{E:remainderdif}
	w_n^{k+1} := w_n^k - T_{x_n^{k+1}} e^{- i {t}_n^{k+1} \langle \nabla \rangle} \L_{{\nu}_n^{k+1}}
	D_{{\lambda}_n^{k+1}}P_n^{k+1} \phi^{k+1}.
\end{equation}
This is a bounded sequence in $H^1$:
\[
\limsup_{n\to\infty} \big\|w_n^{k+1} \big\|_{H^1} \le \limsup_{n\to\infty} \big\|w_n^{k}\big\|_{H^1}.
\]
Furthermore,
\begin{equation}\label{E:k+1thweaklimit1}
	W_{n}^{{k+1},{k+1}}:= D_{{\lambda}_n^{k+1}}^{-1} \L_{{\nu}_n^{k+1}}^{-1} e^{i {t}_n^{k+1} \langle \nabla \rangle}  T_{x_n^{k+1}}^{-1} w_n^{k+1} \rightharpoonup 0 \text{ in } L^2.
\end{equation}
It holds weakly in $H^1$ if $\lambda_n^{k+1}\equiv1$.
Arguing as in Step 2, we have
\EQ{\label{E:decouplingpf2}
\big\|w_n^k\big\|_{L^2}^2={}&\big\|w_n^{k+1}\big\|_{L^2}^2+\normo{T_{x^{k+1}_n} e^{-i t^{k+1}_n \langle \nabla \rangle} \L_{\nu^{k+1}_n}  D_{\lambda^{k+1}_n} P_{n}^{k+1}\phi^{k+1}}_{L^2}^2+o(1),\\
\big\|w_n^k\big\|_{\dot H^1}^2={}&\big\|w_n^{k+1}\big\|_{\dot H^1}^2+\normo{T_{x^{k+1}_n} e^{-i t^{k+1}_n \langle \nabla \rangle} \L_{\nu^{k+1}_n}  D_{\lambda^{k+1}_n} P_{n}^{k+1} \phi^{k+1}}_{\dot H^1}^2+o(1),
}
\begin{align}\label{E:energydcpf2}
\left\|\Re w_n^k  \right\|_{L_x^\frac{2(d+2)}d}^\frac{2(d+2)}d -  \left\|\Re w_n^{k+1} \right\|_{L_x^\frac{2(d+2)}d}^\frac{2(d+2)}d -  \left\|\Re  \left(T_{x^{k+1}_n} e^{-i t^{k+1}_n \langle \nabla \rangle} \L_{\nu^{k+1}_n}  D_{\lambda^{k+1}_n} P_{n}^{k+1} \phi^{k+1} \right)
\right \|_{L_x^\frac{2(d+2)}d}^\frac{2(d+2)}d = o(1),
\end{align}
and
\begin{equation}\label{E:PDsmallnesspf2}
	\normo{w_n^k}_{H^{ \frac12}}^2=\normo{w_n^{k+1}}_{H^{\frac12}}^2+\normo{T_{x^{k+1}_n} e^{-i t^{k+1}_n \langle \nabla \rangle} \L_{\nu^{k+1}_n}  D_{\lambda^{k+1}_n} P_{n}^{k+1} \phi^{k+1}}_{H^{ \frac12}}^2 + o(1)
\end{equation}
as $n\to \infty$.

We repeat the above procedure so long as $\eta \left({\bf w}^{k} \right)>0$.
If $\eta  \left({\bf w}^{k_0} \right)=0$ holds for some $k_0\ge1$, we define $J_0=k_0+1\ge 2$.
Otherwise, let $J_0=\infty$.
Combining \eqref{E:remainderdif} and recalling $w_n^0=v_n$, we obtain the desired decomposition \eqref{eq3.17v173} for all $J \in [1,J_0-1]$.
Similarly, by \eqref{E:decouplingpf1}, \eqref{E:decouplingpf2}, \eqref{E:energydcpf1}, and \eqref{E:energydcpf2},
we have \eqref{eq3.16v162}, \eqref{eq3.17v162} and \eqref{eq3.1v141} for all $J \in [1,J_0-1]$.

 {\bf Step 4. Orthogonality of the parameters.}
Let us establish the mutual orthogonality of the parameters.
We see from \eqref{E:1stweaklimit} and \eqref{E:kthweaklimit1}
that $W_n^{k,k}\rightharpoonup 0$ in $L^2$ for $1\le k \le J_0-1$.
Now we show by induction on $a \in \Z_{>0}$ such that
\begin{equation}\label{E:orthpf1}
W_{k,k+a}:= D_{{\lambda}_n^{k+a}}^{-1}  \L_{{\nu}_n^{k+a}}^{-1} e^{i {t}_n^{k+a} \langle \nabla \rangle}  T_{x_n^{k+a}}^{-1} w_n^k \rightharpoonup \phi^{k+a} \neq0
\end{equation}
holds weakly in $L^2$ for $a\ge1$ and $0\le k \le J_0-1-a$.
If we obtain \eqref{E:orthpf1}, then by means of ``(3)$\Rightarrow$(1)'' of Lemma \ref{L:orthchar}, one obtains the desired
orthogonality of the parameters.

Let us prove \eqref{E:orthpf1}. For simplicity, we consider the case $J_0=\infty$.
The base case $a=1$ follows from \eqref{E:2ndweaklimit1} and \eqref{E:k+1thweaklimit1}.
Pick $a_0 \ge 1$ and suppose that \eqref{E:orthpf1} holds as long as $1 \le a \le a_0$.
Then, by \eqref{E:remainderdif}, we have
\begin{align*}
	W_{k,k+a_0+1}={}& W_{k+1,k+a_0+1} \\&{}+ D_{{\lambda}_n^{k+a_0+1}}^{-1} \L_{{\nu}_n^{k+a_0+1}}^{-1} e^{i {t}_n^{k+a_0+1} \langle \nabla \rangle}  T_{x_n^{k+a_0+1}}^{-1}
\left(T_{x_n^{k+1}} e^{- i {t}_n^{k+1} \langle \nabla \rangle} \L_{{\nu}_n^{k+1}}
	D_{{\lambda}_n^{k+1}}P_n^{k+1} \phi^{k+1} \right).
\end{align*}
By assumption of the induction together with (2) of Lemma \ref{L:orthchar},
one sees that
\[
W_{k+1,k+a_0+1} \rightharpoonup \phi^{k+a_0+1}
\]
and
\[
D_{{\lambda}_n^{k+a_0+1}}^{-1} \L_{{\nu}_n^{k+a_0+1}}^{-1} e^{i {t}_n^{k+a_0+1} \langle \nabla \rangle}  T_{x_n^{k+a_0+1}}^{-1}
\left(T_{x_n^{k+1}} e^{- i {t}_n^{k+1} \langle \nabla \rangle} \L_{{\nu}_n^{k+1}}
	D_{{\lambda}_n^{k+1}}P_n^{k+1} \phi^{k+1} \right)\rightharpoonup 0
\]
weakly in $L^2$ as $n\to\infty$. Since $k$ is arbitrary, we have \eqref{E:orthpf1} for $a=a_0+1$.
Thus, by induction we have \eqref{E:orthpf1} for all $a\ge1$.

 {\bf Step 5. Smallness of the remainder term.}
To complete the proof, we show \eqref{eq3.21v173}.
For this purpose, we first prove
\EQ{\label{E:LPDremainder2}
	\lim_{J\to J_0-1} \eta  \left({\bf w}^J \right) =0.
}
If $J_0$ is finite, then this is true by the definition of $J_0$. Suppose $J_0=\infty$.
Combining \eqref{E:PDsmallnesspf1} and \eqref{E:PDsmallnesspf2}, we have for $J\ge 1$,
\begin{equation}\label{E:PDsmallnesspf3}
\norm{v_n}_{H^{ \frac12}}^2
= \sum_{j=1}^J \normo{T_{x^j_n} e^{-i t^j_n \langle \nabla \rangle} \L_{\nu^j_n}  D_{\lambda^j_n} P_{n}^j \phi^j}_{H^{ \frac12}}^2 + o(1),
 \text{ as $n\to \infty$.}
\end{equation}
Let us claim
\[
\normo{T_{x^j_n} e^{-i t^j_n \langle \nabla \rangle} \L_{\nu^j_n}  D_{\lambda^j_n} P_{n}^j \phi^j}_{H^{ \frac12}} \ge \normo{\tilde{\phi}^j}_{L^2}.
\]
We prove it for $j=1$.
Recalling the definition of the parameters and using the fact that $\L_{\nu}$ is unitary in $H^{\frac12}$, we see that
\[
	\normo{T_{x^1_n} e^{-i t^1_n \langle \nabla \rangle} \L_{\nu^1_n}  D_{\lambda^1_n} P_{n}^1 \phi^1}_{H^{ \frac12}}
	= \normo{D_{\lambda^1_n} P_{n}^1 \phi^1}_{H^{ \frac12}}.
\]
If $\lambda_n^j \equiv 1$, then
\[
\normo{D_{\lambda^1_n} P_{n}^1 \phi^1}_{H^{ \frac12}}
= \normo{\phi^1}_{H^{ \frac12}} = \normo{D_{\lambda_\infty^1} \tilde{\phi}^1}_{H^{ \frac12}}
\ge \normo{D_{\lambda_\infty^1} \tilde{\phi}^1}_{L^2}
= \normo{\tilde{\phi}^1}_{L^2}
\]
If $\lambda_n^j \to \infty$ as $n\to\infty$, then
\[
\normo{D_{\lambda^j_n} P_{n}^j \phi^j}_{H^{ \frac12}} \ge \normo{D_{\lambda^j_n} P_{n}^j \phi^j}_{L^2}
= \normo{\tilde{\phi}^j}_{L^2}.
\]
Hence, the claim follows.
Thus, plugging the identity of the claim to \eqref{E:PDsmallnesspf3}, taking supremum in $n$,
and letting $J\to\infty$, one obtains
\[
\sum_{j=1}^\infty \big\|\tilde{\phi}^j \big\|_{L^2}^2 \le \sup_{n}\normo{v_n}_{H^{ \frac12}} \le A <\infty.
\]
This shows $\big\|\tilde{\phi}^j \big\|_{L^2}\to 0$ as $j\to \infty$. Hence,
\begin{equation}\label{E:PDsmallnesspf4}
	\eta  \left({\bf w}^J  \right) \le 2 \normo{ \phi^{J+1} }_{L^2} \to 0, \text{ as $J \to \infty$.}
\end{equation}
 This is \eqref{E:LPDremainder2}.

If $J_0$ is finite, then \eqref{eq3.21v173} follows from \eqref{E:LPDremainder2},
thanks to Lemma \ref{L:ISE}. Let us consider the case $J_0=\infty$.
Suppose that \eqref{eq3.21v173} fails. Then, there exist $\varepsilon_0>0$ and a sequence
$ \left\{J_k \right\}_{k\ge1}$ with $\lim\limits_{k\to\infty}J_k = \infty$ such that
\[
	\limsup_{n\to\infty} \left\|e^{-it \langle \nabla \rangle} w_n^{J_k} \right\|_{L_{t,x}^\frac{2(d+2)}d} \ge \varepsilon_0
\]
for all $k\ge1$. Together with the bound
\[
	\limsup_{n\to\infty} \normo{w_n^{J_k}}_{H^1} \le \limsup_{n\to\infty} \normo{v_n}_{H^1} \le A,
\]
we see from Lemma \ref{L:ISE} that there exits $\alpha=\alpha \left(M,\varepsilon_0 \right)>0$ such that
$\inf\limits_{k} \eta \left({\bf w}^{J_k} \right) \ge \alpha$. However, this contradicts with \eqref{E:PDsmallnesspf4}.
Thus, we obtain \eqref{eq3.21v173}.
\end{proof}

\section{Low-frequency nonlinear profile: proof of Theorem \ref{th6.2} }\label{s6}
In this section, we will prove Theorem \ref{th6.2}. We study the large scale profile, and using the solution of the mass-critical nonlinear Schr\"odinger equation to approximate the large scale profile. Throughout this section, we write $f(z) = |z|^\frac4d z$.
Before presenting the main result in this section, we first review the global well-posedness and scattering result of the mass-critical nonlinear Schr\"odinger equation
\begin{align}\label{eq6.1}
i\partial_t w+ \frac12 \Delta w  = \mu C_d f(w),
\end{align}
where $\mu = \pm 1$ and the constant $C_d$ is the well-known Wallis integral
\begin{align}\label{eq4.2v91}
C_d : = \frac1{2^{2+ \frac4d}\pi} \int_0^{2\pi} f\left(1 + e^{i\theta}  \right) \,\mathrm{d}\theta
= \frac{\Gamma\left( \frac2d + \frac32 \right) }{  \sqrt{\pi} \Gamma\left( \frac2d + 2\right)} < \frac12.
\end{align}
In particular, we see $C_1 = \frac5{16}$, and $C_2 = \frac38$. For reader's convenience,
we give the computation of \eqref{eq4.2v91} in Appendix A.1. When $\mu = -1$, the ground state solution associated to \eqref{eq6.1} is
\begin{align*}
w_Q(t,x) : = e^{it } \left( \frac1{C_d}\right)^\frac{d}4 Q\left(\sqrt{2}x\right),
\end{align*}
with
\begin{align*}
\left\|w_Q \right\|_{L_x^2}  & =  \left(2C_d \right)^{-\frac{d}4} \| Q\|_{L_x^2},
\end{align*}
where $Q$ is the ground state of \eqref{eq1.4}. For the mass-critical nonlinear Schr\"odinger equation, we have the following result:
\begin{theorem}[Global well-posedness and scattering of the mass-critical NLS, \cite{D3,D1,D2,D4,KTV,KVZ,TVZ0,TVZ}]\label{co6.1}
For any $w_0 \in L_x^2(\mathbb{R}^d)$ and when $\mu = -1$,
we also assume $ \left\|w_0 \right\|_{L_x^2} < \left({2C_d} \right)^{-\frac{d}4}  \|Q\|_{L_x^2} $,
there exists a unique global solution $w$ to \eqref{eq6.1} with $w(0) = w_0$, and
\begin{align*}
\|w\|_{L_{t,x}^\frac{2(d+2)}d(\mathbb{R} \times \mathbb{R}^d)} \le C\left(\|w_0\|_{L_x^2}\right),
\end{align*}
for some continuous function $C$. Moreover, $w$ scatters in $L^2$,
\end{theorem}
We now turn to the proof of Theorem \ref{th6.2}.

\textit{Proof.} By \eqref{eq3.12v161}, we have
\begin{align*}
\phi_n = \L_{\nu_n} T_{\tilde{x}_n} e^{i\tilde{t}_n \langle \nabla \rangle} D_{\lambda_n} P_{\le \lambda_n^\theta} \phi.
\end{align*}
We will take $x_n = \frac{ \nu_n t_n}{ \langle \nu_n \rangle}$ by the spatial translation invariance,
and this leads to $\tilde{x}_n = 0$ and $\tilde{t}_n = \frac{t_n}{ \langle \nu_n \rangle}$.

\textit{Case I. $\nu_n = 0$. }
To show \eqref{eq6.2}, we only need to show
\begin{align}\label{eq6.3}
\left\|v_n(t + t_n, x) -  {e^{-it}}{\lambda_n^{- \frac{d}2}} \psi_\epsilon \left( {\lambda_n^{-2} }{t}, {\lambda_n^{-1} }{x} \right) \right\|_{L_{t,x}^\frac{2(d+2)}d(\mathbb{R}\times \mathbb{R}^d)} < \epsilon.
\end{align}
Before giving the approximate solutions to \eqref{eq2.2}, we first define the solutions to \eqref{eq6.1}, which will be the building block.

When $t_n = 0$, let $w_n$ be the solution to \eqref{eq6.1} with $w_n(0) = P_{\le \lambda_n^\theta} \phi$,
and correspondingly, we let $w_\infty$ be the solution to \eqref{eq6.1} with $w_\infty(0) = \phi$.

In the case when $\frac{t_n}{\lambda_n^2} \to \infty$ (respectively $\frac{t_n}{\lambda_n^2} \to -\infty$), we denote by $w_n$ the solutions to \eqref{eq6.1}, that scatter backward (respectively forward) in time to $e^{it \frac\Delta2} P_{\le \lambda_n^\theta} \phi$. In the same time, we define $w_\infty$ to be the solution to \eqref{eq6.1} that scatters backward (respectively forward) in time to $e^{it \frac\Delta2} \phi$.
By Theorem \ref{co6.1}, we have
\begin{align*}
S_{\mathbb{R}}(w_n) + S_{\mathbb{R}}(w_\infty) \lesssim_{\|\phi\|_{L^2}} 1.
\end{align*}
We also have the following space-time boundedness of the sequence $w_n$ by direct computation, which will be useful later in this section.
\begin{lemma}[Boundedness of the Strichartz norms]\label{le6.4}
The solutions $w_n$ satisfy
\begin{align}\label{eq6.4}
\big\| |\nabla |^s w_n \big\|_{L_t^\infty L_x^2 \cap L_{t,x}^\frac{2(d+2)}d} \lesssim_{\|\phi\|_{L^2}} \lambda_n^{s \theta},
\end{align}
for any $ 0 \le s < 1+ \frac4d$ and
\begin{align}\label{eq5.4v185}
\left\|\langle \nabla \rangle^s \partial_t w_n \right\|_{L_{t,x}^\frac{2(d+2)}d} \lesssim_{\|\phi\|_{L^2}} \lambda_n^{(2+s) \theta}
\end{align}
for any $0 \le s < \frac4d$.
Moreover, we also have the approximation
\begin{align}\label{eq6.6}
\left\|w_n - w_\infty\right\|_{L_t^\infty L_x^2 \cap L_{t,x}^\frac{2(d+2)}d} + \left\|D_{\lambda_n} (w_n - P_{\le \lambda_n^\theta} w_\infty) \right\|_{L_t^\infty H_x^\frac12} \to 0, \text{as } n\to \infty.
\end{align}
\end{lemma}
We can now construct the following approximate solutions to \eqref{eq2.2}:
\begin{align*}
\tilde{v}_n(t) : =
\begin{cases}
e^{-it } D_{\lambda_n}  \left(P_{\le \lambda_n^{2\theta} }w_n \right)\left( \frac{t}{ \lambda_n^2} \right), &  \text{ if } |t| \le  \lambda_n^2 T ,\\
e^{-i \left(t-  \lambda_n^2 T \right)  \langle \nabla \rangle} \tilde{v}_n \left( \lambda_n^2 T \right), & \text{ if } t >  \lambda_n^2 T ,\\
e^{-i \left(t +  \lambda_n^2 T \right) \langle \nabla \rangle} \tilde{v}_n \left(-  \lambda_n^2 T\right), &  \text{ if } t < -  \lambda_n^2 T,
\end{cases}
\end{align*}
where $T$ is a sufficiently large positive number to be specified later. We will show this sequence approximately solves \eqref{eq2.2}, and by invoking Proposition \ref{pr3.4} to deduce that the resulting solutions $v_n$ obey \eqref{eq6.2}. By the Strichartz estimate and Lemma \ref{le6.4}, we have
\begin{align}
\|\tilde{v}_n \|_{L_t^\infty H_x^\frac12 \cap L_{t,x}^\frac{2(d+2)}d}
& \lesssim \|D_{\lambda_n} w_n \|_{L_t^\infty H_x^\frac12} + \left\|D_{\lambda_n} w_n\left(\frac{t}{ \lambda_n^2} \right) \right\|_{L_{t,x}^\frac{2(d+2)}d} \notag \\
& \lesssim_{\|\phi\|_{L_x^2} } 1 + \lambda_n^{- \frac12 } \left\| |\nabla |^\frac12 w_n \right\|_{L_t^\infty L_x^2}
 \lesssim_{ \|\phi\|_{L_x^2} } 1 + \lambda_n^{- \frac{1 - \theta}2   } \lesssim_{ \|\phi\|_{L_x^2} } 1. \notag
\end{align}
By the definition of $\phi_n$ and also \eqref{eq6.6}, we can get
\begin{lemma}[Approximation of the initial data]\label{le6.5}
\begin{align*}
\limsup\limits_{n \to \infty} \left\|\tilde{v}_n(- t_n) - \phi_n \right\|_{H_x^\frac12} \to 0, \text{ as } {T \to \infty}.
\end{align*}
\end{lemma}
Arguing as in \cite{KSV1}, we have $\tilde{v}_n$ are approximate solutions to \eqref{eq2.2} on the large time intervals, by using the solution of the free Schr\"odinger equation to approximate the nonlinear solutions $w_n$ and also the free first order Klein-Gordon propagator is asymptotic small in the Strichartz space. We refer to \cite{KSV1} for similar argument.
\begin{proposition}[Asymptotic small on the large time intervals] \label{pr6.6}
\begin{align*}
& \limsup\limits_{n \to \infty}
 \left( \left\|e^{-i \left(t- \lambda_n^2 T \right) \langle \nabla \rangle } \tilde{v}_n \left( \lambda_n^2 T \right) \right\|_{L_{t,x}^\frac{2(d+2)}d
 \left( \left(\lambda_n^2 T, \infty \right) \times \mathbb{R}^d \right)}
 +  \left\|e^{-i \left(t +  \lambda_n^2 T \right) \langle \nabla \rangle } \tilde{v}_n  \left(-  \lambda_n^2 T \right) \right\|_{L_{t,x}^\frac{2(d+2)}d \left( \left(- \infty, - \lambda_n^2 T \right) \times \mathbb{R}^d \right)} \right)\\
 &    \to 0, \text{ as } T\to \infty.
\end{align*}
\end{proposition}
We now turn to the middle time interval. On the middle time interval, we see $\tilde{v}_n$ satisfies
\begin{align*}
\left(- i \partial_t + \langle \nabla \rangle \right) \tilde{v}_n + \mu \langle \nabla \rangle^{-1}  f\left(\Re \tilde{v}_n \right)
  = e_{1,n} + (e_{2,1,n} +e_{2,2,n}+e_{2,3,n})+e_{3,n},
\end{align*}
where
\begin{align*}
e_{1,n} &:  = e^{-it}
{ \lambda_n^{- \frac{d}2} }\left( P_{\le \lambda_n^{2\theta} } \left( \langle \lambda_n^{-1} \nabla \rangle - 1 + \frac1 {2 \lambda_n^2} \Delta \right)w_n\right)\left( \frac{t}{ \lambda_n^2}, \frac{x}{ \lambda_n}\right),\\
e_{2,1,n}&:   =  \mu \left( \langle \nabla \rangle^{-1} - 1\right) \left( e^{-it}  {C_d }
 P_{\le \lambda_n^{2\theta-1} }
 \left( f\left(w_n
 \left( \frac{t}{\lambda_n^2}, \frac{x}{\lambda_n} \right) \right)\right) \lambda_n^{ -\frac{d}2 - 2} \right),\\
 e_{2,2,n}&:  =  - \mu {C_d  } \lambda_n^{ -\frac{d}2 - 2} e^{-it} \langle \nabla \rangle^{-1}
 (P_{\le \lambda_n^{2\theta-1} }-1) \left(f\left(w_n
 \left( \frac{t}{\lambda_n^2}, \frac{x}{\lambda_n} \right) \right)\right), \\
 e_{2,3,n}&: = - \mu {C_d } \lambda_n^{ -\frac{d}2 - 2} e^{-it} \langle \nabla \rangle^{-1}
 \left( f \left(w_n \left( \frac{t}{\lambda_n^2}, \frac{x}{\lambda_n} \right)\right) - f\left( \left( P_{\le \lambda_n^{2\theta}} w_n \right)
 \left( \frac{t}{\lambda_n^2}, \frac{x}{\lambda_n} \right) \right) \right), \\
e_{3,n}&:  =  \mu \lambda_n^{ -\frac{d}2 - 2} \langle \nabla \rangle^{-1}
\bigg(  f\left(\Re \left( e^{-it} \left( P_{\le \lambda_n^{2\theta}} w_n \right)\left( \frac{t}{\lambda_n^2} , \frac{x}{\lambda_n} \right) \right)  \right)
 - e^{-it}  {C_d }f\left( \left( P_{\le \lambda_n^{2\theta}} w_n \right)\left( \frac{t}{\lambda_n^2}, \frac{x}{\lambda_n} \right) \right)\bigg).
\end{align*}
\begin{remark}
The above decomposition is slightly different from that is used in the previous result \cite{KSV1}.
The point is that we have the factor $\langle \nabla \rangle^{-1}$ in the term $e_{3,n}$, which is crucial when we consider high dimensions.
\end{remark}
By Plancherel's identity, \eqref{eq1.4v138},
 H\"older's inequality and \eqref{eq6.4}, we have
\begin{align}\label{eq6.14}
\left\|e_{1,n} \right\|_{L_t^1 H_x^\frac12 \left( \left[ - \lambda_n^2 T, \lambda_n^2 T \right] \times \mathbb{R}^d \right)}
& = \lambda_n^2 \left\| \langle \lambda_n^{-1} \xi \rangle^\frac12 \left( \langle \lambda_n^{-1} \xi \rangle - 1 - \frac{ |\xi|^2}{2 \lambda_n^2} \right) \widehat{P_{\le \lambda_n^{2\theta}}w_n}(t,\xi) \right\|_{L_t^1 L_\xi^2( [-T,T]\times \mathbb{R}^d)}   \\
& \lesssim \lambda_n^2  \left\| \langle \lambda_n^{-1} \xi \rangle^\frac12 \frac{ |\xi|^4}{ \lambda_n^4} \widehat{P_{\le \lambda_n^{2\theta}}w_n}(t,\xi) \right\|_{L_t^1 L_\xi^2([-T,T] \times \mathbb{R}^d)} \notag\\
& \lesssim T \lambda_n^{-2+8\theta}  \left\|  w_n \right\|_{L_t^\infty L_x^2} \to 0, \text{ as } n \to \infty.\notag
\end{align}
By the Mikhlin multiplier theorem, we obtain
\begin{align}\label{eq6.15}
\quad \left\|\langle \nabla \rangle e_{2,1,n} \right\|_{L_{t,x}^\frac{2(d+2)}{d+4}  \left( \left[ - \lambda_n^2 T, \lambda_n^2 T \right] \times \mathbb{R}^d \right)}
& \lesssim \lambda_n^{ -\frac{d}2 - 2} \left\| \nabla \left( f\left( w_n \left( \frac{t}{\lambda_n^2}, \frac{x}{\lambda_n} \right)\right)
\right) \right\|_{L_{t,x}^\frac{2(d+2)}{d+4} \left( \left[ - \lambda_n^2 T, \lambda_n^2 T \right] \times \mathbb{R}^d \right)}  \\
& \lesssim \lambda_n^{-1}  \|w_n\|_{L_{t,x}^\frac{2(d+2)}d ([-T, T] \times \mathbb{R}^d )}^\frac4d \|\nabla w_n\|_{L_{t,x}^\frac{2(d+2)}d ([-T, T] \times \mathbb{R}^d )}\notag\\
& \lesssim_{\|\phi\|_{L_x^2} } \lambda_n^{-1 + \theta} \to 0, \text{ as } n\to \infty.\notag
\end{align}
Similarly, by the Bernstein inequality, one has
\begin{align}\label{eq:e22n}
\left\|\langle \nabla \rangle e_{2,2,n} \right\|_{L_{t,x}^\frac{2(d+2)}{d+4} \left( \left[ - \lambda_n^2 T, \lambda_n^2 T \right] \times \mathbb{R}^d \right)}
& \lesssim \lambda_n^{ -\frac{d}2 - 1-2\theta} \left\| \nabla \left( f\left( w_n \left( \frac{t}{\lambda_n^2}, \frac{x}{\lambda_n} \right)\right)
\right) \right\|_{L_{t,x}^\frac{2(d+2)}{d+4} \left( \left[ - \lambda_n^2 T, \lambda_n^2 T \right] \times \mathbb{R}^d \right)} \\
&\lesssim_{\|\phi\|_{L_x^2} } \lambda_n^{-\theta} \to 0, \text{ as } n\to \infty\notag
\end{align}
and
\begin{align}\label{eq:e23n}
\left\|\langle \nabla \rangle e_{2,3,n} \right\|_{L_{t,x}^\frac{2(d+2)}{d+4} \left( \left[ - \lambda_n^2 T, \lambda_n^2 T \right] \times \mathbb{R}^d \right)}
& \lesssim  \|w_n\|_{L_{t,x}^\frac{2(d+2)}d([-T, T]\times \mathbb{R}^d )}^\frac4d \|P_{> \lambda_n^{2\theta}} w_n\|_{L_{t,x}^\frac{2(d+2)}d([-T, T]\times \mathbb{R}^d )} \\
& \lesssim \lambda_n^{-2\theta} \|w_n\|_{L_{t,x}^\frac{2(d+2)}d([-T, T] \times \mathbb{R}^d )}^\frac4d \| \nabla w_n\|_{L_{t,x}^\frac{2(d+2)}d([-T, T] \times \mathbb{R}^d )} \notag \\
&\lesssim_{ \|\phi\|_{L_x^2} } \lambda_n^{-\theta} \to 0, \text{ as } n\to \infty. \notag
\end{align}
We now turn to $e_{3,n}$, and show
\begin{align}\label{eq4.12v107}
\left\| \int_0^t e^{-i(t-s) \langle \nabla \rangle} e_{3,n}(s) \,\mathrm{d}s \right\|_{L_t^\infty H_x^\frac12 \cap L_{t,x}^\frac{2(d+2)}d \left( \left[-\lambda_n^2 T, \lambda_n^2 T \right] \times \mathbb{R}^d \right) } \lesssim_T \lambda_n^{-1+8\theta} \to 0, \text{  as $ n \to \infty $. }
\end{align}
For simplicity, we denote $P_{\le \lambda_n^{2\theta}} w_n$ by $w_n$ in what follows.
This will not cause any difference because we do not use the equation for $w_n$ to show \eqref{eq4.12v107}.
We would point out that we do not need the upper bounds on the regularity parameter $s$ in the bounds \eqref{eq6.4} and \eqref{eq5.4v185} any more as long as $\theta$ is replaced by $2\theta$. We have the Fourier series expansion
\begin{equation}\label{E:5.10altpf1}
|\Re u|^\frac4d \Re u = \sum_{k \in \mathbb{Z}} g_{2k-1} |u|^{\frac4d +2 - 2k} u^{2k-1},
\end{equation}
where $g_1 = C_d$ and
\[
	g_{2k-1} := \frac1{2\pi} \int_{-\pi}^\pi |\cos \theta|^{\frac4d} \cos \theta \cos ((2k-1) \theta) d\theta.
\]
By \cite[Proposition A.1]{MMU}, we have
\[
	g_{2k-1} = \frac{ (-1)^{k-1} \Gamma  \left(\frac32 + \frac2d \right) \Gamma \left(k-1-\frac2d \right) }{ \sqrt{\pi} \Gamma \left(-\frac2d \right) \Gamma  \left( k + 1+ \frac2d \right) }
= O \left(|k|^{-\frac4d - 2} \right), \text{ as $|k|\to \infty$. }
\]
The expansion \eqref{E:5.10altpf1} yields another formula for the error term
\[ e_{3,n} = \sum\limits_{k\in \mathbb{Z},\, k\ne 1} e_{3,k,n},
\]
where
\begin{align*}
e_{3,k,n} = \mu g_{2k-1} \lambda_n^{-\frac{d}2-2} e^{-i (2k-1)t} \langle \nabla \rangle^{-1} \left( \left|w_n\left({\lambda_n^{-2}}{t}, {\lambda_n^{-1}} {x} \right)\right|^{\frac4d + 2-2k} w_n\left( {\lambda_n^{-2}}{t},{\lambda_n^{-1}} {x} \right)^{2k-1}
\right).
\end{align*}
Let us introduce $f_{k,n}$ defined by
\begin{align*}
f_{k,n}(t) = -i \int_0^t e^{-i(t-s) \langle \nabla \rangle} e_{3,k,n}(s) \,\mathrm{d}s.
\end{align*}
Remark that what we want to estimate is the ${L_t^\infty H_x^\frac12 \cap L_{t,x}^\frac{2(d+2)}d \left( \left[-\lambda_n^2 T, \lambda_n^2 T \right] \times \mathbb{R}^d \right) }$ norm of $f_{n} := \sum\limits_{k\ne 1} f_{k,n}$.
A computation shows that
\[
	\left(-i \partial_t + \langle \nabla \rangle \right) f_{k,n} = - e_{3,k,n}
\]
and
\begin{align*}
& (-i \partial_t + \langle \nabla \rangle) e_{3,k,n} \\
& = -2(k-1)e_{3,k,n} -i \mu g_{2k-1} \lambda_n^{-\frac{d}2 - 4} e^{-i(2k-1)t} \left( \langle \lambda_n^{-1} \nabla \rangle^{-1} \partial_t \left( |w_n|^{\frac4d + 2-2k} w_n^{2k-1} \right)\right)\left(\frac{t}{\lambda_n^2}, \frac{x}{\lambda_n}\right )\\
& \quad + \mu g_{2k-1} \lambda_n^{- \frac{d}2 - 2} e^{-i(2k-1)t} \left( \langle \lambda_n^{-1} \nabla \rangle^{-1} \left( \langle \lambda_n^{-1} \nabla \rangle - 1\right)\left(|w_n|^{\frac4d + 2 - 2k } w_n^{2k-1} \right)\right)\left(\frac{t}{\lambda_n^2}, \frac{x}{\lambda_n}\right).
\end{align*}
Combining these two identities, one obtains
\begin{align*}
& (-i \partial_t + \langle \nabla \rangle) \left(f_{k,n}- \frac1{2(k-1)} e_{3,k,n}\right) \\
& = \frac{i \mu g_{2k-1}}{2(k-1)} \lambda_n^{-\frac{d}2 - 4} e^{-i(2k-1)t} \left( \langle \lambda_n^{-1} \nabla \rangle^{-1} \partial_t \left( |w_n|^{\frac4d + 2-2k} w_n^{2k-1} \right)\right)\left(\frac{t}{\lambda_n^2}, \frac{x}{\lambda_n}\right )\\
& \quad - \frac{\mu g_{2k-1}}{2(k-1)} \lambda_n^{- \frac{d}2 - 2} e^{-i(2k-1)t} \left( \langle \lambda_n^{-1} \nabla \rangle^{-1} \left( \langle \lambda_n^{-1} \nabla \rangle - 1\right)\left(|w_n|^{\frac4d + 2 - 2k } w_n^{2k-1} \right)\right)\left(\frac{t}{\lambda_n^2}, \frac{x}{\lambda_n}\right).
\end{align*}
By the Strichartz estimate, one has the desired estimate
\begin{align*}
\|f_{n} \|_{L_t^\infty H^\frac12_x \cap L_{t,x}^\frac{2(d+2)}d \left( \left[- \lambda_n^2 T, \lambda_n^2 T \right] \times \mathbb{R}^d \right) } \lesssim_T \lambda_n^{-1+6\theta},
\end{align*}
which is exactly \eqref{eq4.12v107}, from the following four estimates:
\begin{align*}
 \|e_{3,k,n} \|_{L^\infty_t   H^\frac12_x \left( \left[-\lambda_n^2 T, \lambda_n^2 T \right] \times \mathbb{R}^d \right)}
 &{}\lesssim |g_{2k-1}| \lambda_n^{-\frac{d}2-2} \left\| |w_n|^{1+\frac4d}(\cdot, \lambda_n^{-1} \cdot) \right\|_{L_t^\infty L_x^2(\mathbb{R}\times \mathbb{R}^d)} \\
 &{}= |g_{2k-1}| \lambda_n^{-2} \left\| w_n \right\|_{L_t^\infty L_x^{2(1+\frac4d)}([-T,T]\times \mathbb{R}^d)}^{1+\frac4d} \\
 &{}\lesssim
\langle k \rangle^{- \frac4d - 2} \lambda_n^{-2} \|w_n\|_{L_t^\infty H_x^{\frac{2d}{d+4}}([-T,T]\times \mathbb{R}^d)}^{1 + \frac4d} \\
&{}\lesssim \langle k \rangle^{- \frac4d - 2} \lambda_n^{-2+4\theta},
\end{align*}
\begin{align*}
 \|e_{3,k,n} \|_{L_{t,x}^\frac{2(d+2)}d ( [-\lambda_n^2 T, \lambda_n^2 T ] \times \mathbb{R}^d )}
 & \lesssim |g_{2k-1} | \lambda_n^{-\frac{d}2-2} \left\| |w_n|^{1+\frac4d} (\lambda_n^{-2} \cdot, \lambda_n^{-1} \cdot) \right\|_{L_{t,x}^\frac{2(d+2)}d ([-\lambda_n^2 T, \lambda_n^2 T ] \times \mathbb{R}^d) } \\
 & = |g_{2k-1} | \lambda_n^{-2} \left\| w_n\right\|_{L_{t,x}^{\frac{2(d+2)(d+4)}{d^2}} ([-T,T]\times \mathbb{R}^d)}^{1+\frac4d} \\
 & \lesssim \langle k \rangle^{-\frac4d -2 } \lambda_n^{-2} T^{\frac{d}{2(d+2)}} \|w_n\|_{L^\infty_t H_x^{\frac{d(3d+4)}{(d+2)(d+4)}}([-T,T]\times \mathbb{R}^d)}^{1+\frac4d }\\
 &{}\lesssim_T \langle k \rangle^{- \frac4d - 2} \lambda_n^{-2+6\theta},
\end{align*}
\begin{align*}
& \left\|\frac{i \mu g_{2k-1} }{2(k-1)} \lambda_n^{-\frac{d}2 - 4} e^{-i (2k-1)t} \langle \nabla \rangle^{-1} \left(  \partial_t \left(|w_n|^{\frac4d + 2 - 2k} w_n^{2k-1} \right)\right)\left( \frac{\cdot}{\lambda_n^2}, \frac{\cdot}{\lambda_n} \right) \right\|_{L^1_t L^2_x( [-\lambda_n^2 T, \lambda_n^2 T ] \times \mathbb{R}^d )}\\
&\quad \lesssim |g_{2k-1} | \lambda_n^{-2} \left\||w_n|^\frac4d |\partial_t w_n| \right\|_{L^1_t L^2_x([-T,T]\times \mathbb{R}^d)} \\
&\quad
\lesssim \langle k \rangle^{-\frac4d - 2} \lambda_n^{-2} \|w_n\|_{L^\infty_t H^{\frac{2d}{d+4}}_x([-T,T]\times \mathbb{R}^d)}^{ \frac4d ([-T,T]\times \mathbb{R}^d)} \|\partial _t w_n\|_{L^\infty_t H^{\frac{2d}{d+4}}_x([-T,T]\times \mathbb{R}^d)}\\
&\quad
\lesssim \langle k \rangle^{-\frac4d - 2} \lambda_n^{-2+6\theta},
\end{align*}
and
\begin{align*}
& \left\|\frac{\mu g_{2k-1}}{2(k-1)} \lambda_n^{-\frac{d}2 - 2} e^{-i (2k-1)t} \left( \left( \langle \lambda_n^{-1} \nabla \rangle - 1\right)\left ( |w_n|^{\frac4d + 2-2k} w_n^{2k-1}\right)\right) \left( \frac{\cdot}{\lambda_n^2}, \frac{\cdot}{\lambda_n}\right ) \right\|_{L^1_t L^2_x( [-\lambda_n^2 T, \lambda_n^2 T ] \times \mathbb{R}^d )} \\
& \quad \lesssim \frac{ |g_{2k-1} |}{ |k-1|} \left\| \lambda_n^{-1} \nabla \left( |w_n|^{\frac4d + 2-2k} w_n^{2k - 1}\right)  \right\|_{L^1_t L^2_x([-T,T]\times \mathbb{R}^d)}\\
&\quad \lesssim \langle k \rangle^{-\frac4d - 2} \lambda_n^{-1} \|w_n\|_{L^\infty_t H^{\frac{2d}{d+4}}_x([-T,T]\times \mathbb{R}^d)}^{ \frac4d} \|\nabla w_n\|_{L^\infty_t H^{\frac{2d}{d+4}}_x([-T,T]\times \mathbb{R}^d)} \\
&\quad \lesssim \langle k \rangle^{-\frac4d - 2} \lambda_n^{-1+6\theta}.
\end{align*}
We have used the elementary estimate $|\frac{d}{dz} (|z|^{\frac4d + 2 - 2k} z^{2k-1})| \lesssim \langle k\rangle |z|^{\frac4d}$ to obtain the third and fourth estimates.
Notice that the decay in $k$ is enough to sum up.
Therefore, \eqref{eq4.12v107} follows. After the above computation, we have
\begin{proposition}\label{pr6.10}
For any $\epsilon > 0$, there exist sufficiently large positive constants $T$ and $N$, such that for any $n \ge N$, $\tilde{v}_n$ satisfy
\begin{align*}
\left(-i \partial_t + \langle \nabla \rangle \right)  {\tilde{v}}_n = -  \mu \langle \nabla \rangle^{-1} f\left(\Re  {\tilde{v}}_n  \right)
+ \tilde{e}_{1,n} + \tilde{e}_{2,n} +  \tilde{e}_{3,n},
\end{align*}
with the error terms $\tilde{e}_{1,n}$, $ \tilde{e}_{2,n}$, $ \tilde{e}_{3,n}$ small in the sense that
\begin{align*}
\left\|\tilde{e}_{1,n} \right\|_{L_t^1 H_x^\frac12(\mathbb{R} \times \mathbb{R}^d) } + \left\|\langle \nabla \rangle
\tilde{e}_{2,n} \right\|_{L_{t,x}^\frac{2(d+2)}{d+4}(\mathbb{R} \times \mathbb{R}^d)} + \left\|\int_0^t e^{-i(t-s) \langle \nabla \rangle} \tilde{e}_{3,n}(s) \,\mathrm{d}s \right\|_{L_t^\infty H_x^\frac12 \cap L_{t,x}^\frac{2(d+2)}d(\mathbb{R}  \times \mathbb{R}^d)}  \le \epsilon.
\end{align*}
\end{proposition}
\begin{proof}
On the interval $ \left[ - \lambda_n^2 T, \lambda_n^2 T \right]$, we can take
\begin{align*}
\tilde{e}_{1,n} = e_{1,n}, \quad \tilde{e}_{2,n} = e_{2,1,n}+ e_{2,2,n} + e_{2,3,n},
\quad
\tilde{e}_{3,n} = e_{3,n}.
\end{align*}
By \eqref{eq6.14}, \eqref{eq6.15}, \eqref{eq:e22n} and \eqref{eq:e23n}, we have
\begin{align*}
\left\|\tilde{e}_{1,n} \right\|_{L_t^1 H_x^\frac12 \left( \left[ - \lambda_n^2 T, \lambda_n^2 T \right] \times \mathbb{R}^d \right)} + \left\|\langle \nabla \rangle \tilde{e}_{2,n} \right\|_{L_{t,x}^\frac{2(d+2)}{d+4}  \left( \left[ - \lambda_n^2 T, \lambda_n^2 T \right] \times \mathbb{R}^d \right)} \lesssim_T \lambda_n^{- 2 + 8 \theta} + \lambda_n^{-1+ \theta}+ \lambda_n^{- \theta}.
\end{align*}
Together with \eqref{eq4.12v107}, $\forall\, T > 0$, we can take $N$ large enough, such that for each $n \ge N$,
\begin{align*}
& \left\|\tilde{e}_{1,n} \right\|_{L_t^1 H_x^\frac12 \left( \left[ - \lambda_n^2 T, \lambda_n^2 T \right] \times \mathbb{R}^d \right)}+ \left\|\tilde{e}_{2,n} \right\|_{L_{t,x}^\frac{2(d+2)}{d+4} \left( \left[ - \lambda_n^2 T, \lambda_n^2 T \right]\times \mathbb{R}^d \right)} \\
 & \quad + \left\|\int_0^t e^{-i(t-s) \langle \nabla \rangle} \tilde{e}_{3,n}(s) \,\mathrm{d}s \right\|_{L_t^\infty H_x^\frac12 \cap
L_{t,x}^\frac{2(d+2)}d \left( \left[ - \lambda_n^2 T, \lambda_n^2 T \right] \times \mathbb{R}^d \right)}
\le \frac\epsilon2.
\end{align*}
We now turn to the time intervals $ \left(-\infty, - \lambda_n^2 T \right) \cup  \left(\lambda_n^2 T, \infty \right)$. In this case, we choose $\tilde{e}_{1,n} = \tilde{e}_{2,n} = 0$ and $\tilde{e}_{3,n} = \mu \langle \nabla \rangle^{-1} f\left( \Re  {\tilde{v}}_n  \right)$. By Proposition \ref{pr6.6}, \eqref{eq4.12v107} and the Strichartz estimate, for $T$ and $n$ sufficiently large, one has
\begin{align*}
\left\|\int_0^t e^{-i(t-s) \langle \nabla \rangle} \tilde{e}_{3,n}(s) \,\mathrm{d}s \right\|_{L_t^\infty H_x^\frac12 \cap
L_{t,x}^\frac{2(d+2)}d \left(|t| \ge \lambda_n^2 T \right) }
\lesssim \left\| \tilde{v}_n \right\|_{L_{t,x}^\frac{2(d+2)}d \left(|t| \ge T\lambda_n^2 \right)}^{\frac4d + 1} \le \frac\epsilon2.
\end{align*}
This completes the proof of the Proposition.
\end{proof}
By Lemma \ref{le6.5}, Proposition \ref{pr6.10}, and Proposition \ref{pr3.4}, we can obtain a solution $v_n$ to \eqref{eq2.2} with $v_n(0) = \phi_n$, for $n$ large enough. Moreover,
\begin{align}\label{eq6.19}
 \left\|v_n(t) - \tilde{v}_n(t-t_n) \right\|_{L_t^\infty H_x^\frac12 \cap L_{t,x}^\frac{2(d+2)}d}
\to 0, \text{ as } n\to \infty.
\end{align}
We now turn to the proof of \eqref{eq6.3}. By density, we can take $\psi_\epsilon \in C_c^\infty ( \mathbb{R} \times \mathbb{R}^d) $ such that
\begin{align}\label{eq5.13v84}
\left\|e^{-it} D_{\lambda_n} \left( \psi_\epsilon  \left( \lambda_n^{-2} t \right) - w_\infty( \lambda_n^{-2} t) \right) \right\|_{L_{t,x}^\frac{2(d+2)}d}
= \left\|\psi_\epsilon - w_\infty \right\|_{L_{t,x}^\frac{2(d+2)}d} < \frac\epsilon2.
\end{align}
By the definition of $\tilde{v}_n$, the triangle inequality, Proposition \ref{pr6.6}, \eqref{eq6.6}, the dominated convergence theorem, we have by taking $T$ sufficiently large and $n$ large enough,
\begin{align*}
& \left\|\tilde{v}_n(t)  - e^{-it} D_{\lambda_n} w_\infty \left( \lambda_n^{-2} t \right) \right\|_{L_{t,x}^\frac{2(d+2)}d}\\
\lesssim  & \left\|\tilde{v}_n \right\|_{L_{t,x}^\frac{2(d+2)}d \left( \left\{ |t| > T \lambda_n^2 \right\} \times \mathbb{R}^d \right)}
+ \left\|w_n - w_\infty \right\|_{L_{t,x}^\frac{2(d+2)}d} + \left\|w_\infty \right\|_{L_{t,x}^\frac{2(d+2)}d
\left( \left\{|t|>T \right\} \times \mathbb{R}^d \right)} < \frac\epsilon2.
\end{align*}
Combining this with \eqref{eq6.19} and \eqref{eq5.13v84}, we get \eqref{eq6.2} when $\nu_n = 0$.

\textit{Case II. $\nu_n \to \nu \in \mathbb{R}^d$, as $n\to \infty$.}

By the proof in \textit{Case I}, there is a global solution $v_n^0$ to \eqref{eq2.2} with
\begin{align*}
v_n^0(0) = T_{\tilde{x}_n} e^{i\tilde{t}_n \langle \nabla \rangle} D_{\lambda_n} P_{\le \lambda_n^\theta} \phi,
\end{align*}
for $n$ large enough.
Moreover, $S_{\mathbb{R}}(v_n^0) \lesssim_{ \|\phi\|_{L_x^2} } 1$ and for any $\epsilon > 0$,
there exists $\psi_\epsilon^0 \in C_c^\infty(\mathbb{R} \times \mathbb{R}^d)$ and $N_\epsilon^0$ so that
\begin{align}\label{eq6.22}
\left\|\Re \left( v_n^0\left(t + \tilde{t}_n, x + \tilde{x}_n\right) - { \lambda_n^{- \frac{d}2} }{e^{-it}} \psi_\epsilon^0\left( { \lambda_n^{-2}}{t}, { \lambda_n^{-1} }{x} \right) \right) \right\|_{L_{t,x}^\frac{2(d+2)}d} < \epsilon,
\end{align}
when $n \ge N_\epsilon^0$. Before continuing, we recall the following results, which are extension of the results in \cite{KSV1} in higher dimensions.
Arguing as in \cite{KSV1}, by the finite speed of propagation, we have
\begin{lemma}\label{le3.5}
For any $(u_0,u_1) \in H_x^1 \times L_x^2$, there exist sufficiently small constant $\epsilon > 0$ and a local solution $u$ defined in $\Omega = \left\{(t,x)\in \mathbb{R} \times \mathbb{R}^d: |t| - \epsilon |x| < \epsilon\right\}$ to \eqref{eq1.1} with $(u(0), \partial_t u(0)) = (u_0,u_1)$. In addition, the solution $u$ satisfies
\begin{align}\label{eq3.4}
 \sup\limits_{|t| < \epsilon R} \int_{|x| > R} \left( | \partial_t u(t,x)|^2 + |\nabla u(t,x)|^2 + |u(t,x)|^2\right) \,\mathrm{d}x \to 0, \text{ as }  {R\to \infty}.
\end{align}
\end{lemma}
\begin{lemma}\label{co3.7}
Given $\left(u(0), \partial_t u(0)\right)  \in H^1\times L^2$ and $\frac{|\nu |}{\langle \nu \rangle } < \epsilon$ for some $\epsilon > 0$, we have
$u \circ L_\nu$ is a solution to \eqref{eq1.1} on $(- \epsilon, \epsilon) \times \mathbb{R}^d$ and $\left(u \circ L_\nu(0,x),
(u \circ L_\nu)_t(0,x) \right)\in H^1 \times L^2$ is continuous with respect to $\nu$.
\end{lemma}
We can now return to the proof when $\nu_n \to \nu \in \mathbb{R}^d$, as $n\to \infty$, we have the following extension of Proposition 6.11 in \cite{KSV1} in higher dimensional case. Although the proof is a slight modification of Proposition 6.11 of \cite{KSV1}, we present the proof for self-contained.
\begin{proposition}[Matching initial data]\label{pr6.11}
For $n$ large enough, the global solution
\begin{align*}
v_n^1:= \left( 1 + i \langle \nabla \rangle^{-1} \partial_t \right) \Re \left( v_n^0 \circ L_{\nu_n}\right)
\end{align*}
of \eqref{eq2.2} satisfies $\sup\limits_n S_{\mathbb{R}}(v_n^1) \lesssim_{ \|\phi\|_{L_x^2} } 1$ and
\begin{align}\label{eq5.21v140}
\left\|v_n^1(0) - \phi_n \right\|_{H_x^1} \to 0, \text{ as } n\to \infty.
\end{align}
\end{proposition}
\begin{proof}
We have the decomposition
\begin{align*}
\Re v_n^0= u_n^{0, l } + \tilde{u}_n^0,
\end{align*}
where $u_n^{0, l }$ is the solution of the free Klein-Gordon equation with
\begin{align*}
\left(\left( 1 +i \langle \nabla \rangle^{-1} \partial_t \right) u_n^{0,l }\right) (0) = v_n(0) = L_{\nu_n}^{-1} \phi_n.
\end{align*}
By \eqref{eq2.13v139}, we have
\begin{align*}
\left(\left( 1 + i \langle \nabla \rangle^{-1} \partial_t \right) \left( u_n^{0, l } \circ L_{\nu_n} \right)\right) (0) = \phi_n,
\end{align*}
we can then obtain $\left\|v_n^1(0) - \phi_n \right\|_{H_x^1}  = \left\|\tilde{u}_n^0 \circ L_{\nu_n}(0, \cdot) \right\|_{H_x^1}$.

By direct calculation, we see $\tilde{u}_n^0$ obeys
\begin{align*}
\begin{cases}
\partial_t^2 \tilde{u}_n^0 - \Delta \tilde{u}_n^0 + \tilde{u}_n^0 = - \mu | \Re v_n^0|^\frac4d \Re v_n^0,\\
 \tilde{u}_n^0(0,x) = \partial_t \tilde{u}_n^0(0,x) = 0.
\end{cases}
\end{align*}
On the space-time domain $\Omega = \left\{(t,x)\in \mathbb{R} \times \mathbb{R}^d: |t| - \epsilon |x| < \epsilon\right\}$, by Lemma \ref{le3.5} and the Strichartz estimate, we have
\begin{align*}
\left\|\tilde{u}_n^0 \right\|_{L_t^q L_x^r(\Omega)} + \left\|\nabla_{t,x} \tilde{u}_n^0 \right\|_{L_t^\infty L_x^2(\Omega)} < \infty, \text{ for any sharp Schr\"odinger admissible pair } (q,r).
\end{align*}
Since $\Re v_n^0 $ satisfies \eqref{eq3.4}, and the analogous estimate for $u_n^{0,l }$ follows from finite speed of propagation and energy conservation, this yields
\begin{align}\label{eq3.10}
 \sup\limits_{|t| \le \epsilon R} \int_{|x| > R}   \left|\partial_t \tilde{u}_n^0(t,x) \right|^2 +  \left|\nabla \tilde{u}_n^0(t,x) \right|^2 + \left|\tilde{u}_n^0(t,x) \right|^2   \,\mathrm{d}x
\to 0, \text{ as } R\to \infty.
\end{align}
Let $\mathcal{T}$ be the stress energy tensor of $\tilde{u}_n^0$, its components are
\begin{align*}
& \mathcal{T}^{00} = \frac12 \left| \partial_t \tilde{u}_n^0 \right|^2 + \frac12 \left|\nabla \tilde{u}_n^0 \right|^2 + \frac12 \left| \tilde{u}_n^0 \right|^2,
 \ \mathcal{T}^{0j} = \mathcal{T}^{j0} = - \partial_t \tilde{u}_n^0 \partial_j \tilde{u}_n^0,\\
\text{ and } &
 \mathcal{T}^{jk} = \partial_j \tilde{u}_n^0  \partial_k \tilde{u}_n^0 - \delta_{jk} \left(\mathcal{T}^{00} -  \left|\partial_t \tilde{u}_n^0  \right|^2 \right),
\end{align*}
where $j, k \in \{1, \cdots, d\}$. Let the vector $\mathbf{p}_n$ with components defined by
\begin{align*}
{p}_n^\alpha = \langle \nu_n \rangle \mathcal{T}^{0\alpha} + \nu_{n,1} \mathcal{T}^{1 \alpha} + \nu_{n,2} \mathcal{T}^{2 \alpha} + \cdots + \nu_{n,d} \mathcal{T}^{d \alpha} , \ \alpha \in \{0,1,2, \cdots, d\}.
\end{align*}
By direct computation, we have
\begin{align}\label{eq2.31v77}
\nabla_{t,x} \cdot \mathbf{p}_n = - \mu  \left|\Re v_n^0 \right|^\frac4d \Re v_n^0  \left( \langle \nu_n \rangle \partial_t \tilde{u}_n^0 - \nu_n \cdot \nabla_x \tilde{u}_n^0 \right),
\end{align}
and by Gauss' formula,
\begin{align*}
\int_{L_{\nu_n} (t, \mathbb{R}^d)} \mathbf{p}_n \cdot \mathrm{d}\mathbf{S}
& =\int_{\mathbb{R}^d} \Big( \langle \nu_n \rangle p_n^0 +  \sum\limits_{j = 1}^d \nu_{n,j } p_n^j\Big) \circ L_{\nu_n} (t,x) \,\mathrm{d}x\\
& = \frac12 \int_{\mathbb{R}^d} \left|\partial_t \left( \tilde{u}_n^0 \circ L_{\nu_n}  \right)\right|^2 + \left|\nabla  \left(\tilde{u}_n^0 \circ L_{\nu_n} \right)\right|^2 + \left|\tilde{u}_n^0 \circ L_{\nu_n} \right|^2 \,\mathrm{d}x,
\end{align*}
where $\mathrm{d} \mathbf{S}$ is the surface measure times the unit normal vector.

We now consider the estimate of the nonlinearity in
\begin{align*}
\Omega_n = \left\{ (t,x): \left( \langle \nu_n \rangle t + \nu_n \cdot x \right) t < 0\right\}.
\end{align*}
For any $(t,x) \in \mathbb{R} \times \mathbb{R}^d$, denote
\begin{align*}
\psi_R(t,x) = \phi\left( \frac{|t|+ |x|}R\right),
\end{align*}
where $\psi$ is the cut-off function defined in \eqref{eq3.35v182},
by applying the divergence theorem together with \eqref{eq3.10}, \eqref{eq2.31v77}, and Lemma \ref{le3.5}, we have
\begin{align}\label{eq5.25v140}
\frac12 \left\|\tilde{u}_n^0 \circ L_{\nu_n}(0, \cdot) \right\|_{H_x^1}^2
\le & \lim\limits_{R\to \infty} \frac12 \int_{\mathbb{R}^d} \left(\left|\partial_t (\tilde{u}_n^0 \circ L_{\nu_n} )\right|^2 + |\nabla \left( \tilde{u}_n^0 \circ L_{\nu_n} \right)|^2 + \left|\tilde{u}_n^0 \circ L_{\nu_n}  \right|^2 \right) \psi_R \,\mathrm{d}x \notag \\
\le & \limsup\limits_{R\to \infty} \iint_{\Omega_{t,\nu_n }} \left|\psi_R \nabla_{t,x} \cdot \mathbf{p}_n \right| + \left|\mathbf{p}_n \cdot \nabla_{s,y} \psi_R\right|\,\mathrm{d}y \mathrm{d}s \notag \\
\le & \iint_{\Omega_{t,\nu_n }} \left|\nabla_{t,x} \cdot \mathbf{p}_n \right| + \limsup\limits_{R\to \infty} \frac1R \int_{-\epsilon R}^{\epsilon R} \int_{|x|\sim R} \left|\langle \nabla_{t,x} \rangle \tilde{u}_n^0 \right|^2 \,\mathrm{d}x \mathrm{d}t \notag\\
 = &  \iint_{\Omega_n}  \left|\nabla_{t,x} \cdot \mathbf{p}_n \right| \,\mathrm{d}x \mathrm{d}t
 \lesssim   \left\|\Re v_n^0 \right\|_{L_{t,x}^\frac{2(d+2)}d(\Omega_n)}^\frac{d+4}d \left\|\nabla_{t,x} \tilde{u}_n^0 \right\|_{L_{t,x}^\frac{2(d+2)}d(\mathbb{R}\times \mathbb{R}^d)},
\end{align}
where
$\Omega_{t, \nu_n} := \left\{ (s,y): \left( \langle \nu_n \rangle^{-1}  \left(t- \nu_n \cdot y \right) - s \right) s > 0\right\}$.

We now estimate the right hand side of \eqref{eq5.25v140}. We can see $\forall\, \psi \in C_c^\infty$,
\begin{align*}
\int_{\Omega_n} \left| \lambda_n^{- \frac{d}2} \psi\left( \frac{t- \tilde{t}_n}{ \lambda_n^2}, \frac{x - \tilde{x}_n} { \lambda_n} \right) \right|^\frac{2(d+2)}d \,\mathrm{d}x \mathrm{d}t
\lesssim \lambda_n^{-1}\|\psi\|_{L_{t,x}^\infty} \to 0, \text{ as } n\to \infty.
\end{align*}
This together with \eqref{eq6.22} and the triangle inequality yields for $n$ sufficiently large,
\begin{align}\label{eq5.27v140}
\left\|\Re v_n^0\right\|_{L_{t,x}^\frac{2(d+2)}d ( \Omega_n)} \to 0, \text{ as } n\to \infty.
\end{align}
By the triangle inequality, \eqref{eq3.3}, $S_{\mathbb{R}}(v_n^0) \lesssim_{ \|\phi\|_{L_x^2} } 1$, and Strichartz, we get
\begin{align}\label{eq5.26v140}
 \left\|\nabla_{t,x} \tilde{u}_n^0 \right\|_{L_{t,x}^\frac{2(d+2)}d} & \le \left\|\nabla_{t,x} \Re v_n^0 \right\|_{L_{t,x}^\frac{2(d+2)}d} + \left\|\nabla_{t,x} u_n^{0, l } \right\|_{L_{t,x}^\frac{2(d+2)}d}\\
 & \lesssim_{ \|\phi \|_{L_x^2} }
  \left\| \langle \nabla \rangle^\frac32 D_{\lambda_n} P_{\le \lambda_n^\theta} \phi\right\|_{L_x^2} + \left\|v_n^0(0) \right\|_{H_x^\frac32}
\lesssim_{ \|\phi \|_{L_x^2} } 1. \notag
\end{align}
By \eqref{eq5.25v140}, \eqref{eq5.26v140}, and \eqref{eq5.27v140}, we can finish the proof of \eqref{eq5.21v140}.
\end{proof}
Since $v_n^0$ is a solution of \eqref{eq2.2}, $ \Re \left( v_n^0\circ L_{\nu_n}\right)$ solves \eqref{eq1.1} by Lemma \ref{co3.7}. In general, $v_n^0 \circ L_{\nu_n}$ is not a solution of \eqref{eq2.2}, and also
\begin{align*}
v_n^1:= \left( 1 + i \langle \nabla \rangle^{-1} \partial_t \right) \Re \left( v_n^0 \circ L_{\nu_n}\right)
\end{align*}
solves \eqref{eq2.2} with $S_{\mathbb{R}}(v_n^1) = S_{\mathbb{R}}(v_n^0)$, which equals to $v_n^0 \circ L_{\nu_n}$ only when $\nu_n = 0$. Thus it is necessary to pass through real solutions here. By Proposition \ref{pr6.11}, the difference between $v_n^1(0)$ and $v_n(0)$ is small. By Proposition \ref{pr6.11} and Proposition \ref{pr3.4}, there exists a global solution $v_n$ to \eqref{eq2.2} with $v_n(0) = \phi_n$ and $S_{\mathbb{R}}(v_n) \lesssim_{\|\phi\|_{L_x^2} } 1$ for $n$ large enough. Moreover,
\begin{align*}
\left\|\Re \left( v_n - v_n^1 \right) \right\|_{L_{t,x}^\frac{2(d+2)}d} \to 0, \text{ as } n\to \infty.
\end{align*}
This together with $\Re v_n^0 = \Re \left( v_n^1 \circ L_{\nu_n}^{-1}\right) $ and \eqref{eq6.22} shows \eqref{eq6.2}.

\appendix
\section{}
In this appendix, we give the detail of the calculation of \eqref{eq4.2v91} and another proof of the important estimate \eqref{eq4.12v107}.
\subsection{The calculation of \eqref{eq4.2v91}}
We now compute the integral
\begin{align*}
C_d = \frac1{\pi 2^{2 + \frac4d}} \int_0^{2\pi} \left| 1 +e^{i \theta} \right|^\frac4d \left( 1 + e^{i \theta} \right) \,\mathrm{d} \theta.
\end{align*}
By easy calculation, we have
\begin{align*}
  \frac1{ \pi 2^{2+ \frac4d}} \int_0^{2\pi} \left| 1 + e^{i \theta} \right|^\frac4d \left( 1 + e^{i\theta}\right) \,\mathrm{d}\theta
=  & \frac{2^\frac2d} { \pi 2^{2 + \frac4d}} \int_{-\pi}^{\pi} ( 1 + \cos\theta)^\frac2d ( 1 + \cos\theta) \,\mathrm{d}\theta\\
= & \frac{1} {2\pi } \int_{0}^{2\pi}  \left| \cos \left(\frac\theta 2\right) \right|^\frac4d \cos^2 \left( \frac\theta 2\right) \,\mathrm{d}\theta,
\end{align*}
where we have used the fact that an integral of an odd function on the interval $[-\pi,\pi]$ is zero.
We have
\begin{align*}
 \frac{1} { 2\pi } \int_{0}^{2\pi}  \left|\cos \left(\frac\theta 2\right)\right|^\frac4d \cos^2 \left(\frac\theta 2\right) \,\mathrm{d}\theta
= &  \frac1{\pi} \int_{0}^{\pi} \left|\cos \theta \right|^\frac4d \cos^2 \theta \,\mathrm{d}\theta \\
= & \frac1{2\pi} \int_{-\pi}^\pi \left|\cos \theta\right|^{ ( \frac4d + 1) - 1} \cos\theta\cos \theta \,\mathrm{d}\theta
= \frac1{\sqrt{\pi}} \frac{\Gamma \left( \frac2d + \frac32 \right)  }{ \Gamma \left( \frac2d + 2  \right)},
\end{align*}
where we use the Proposition A. 1 in \cite{MMU}. Therefore,
\begin{align*}
C_d = \frac1{\sqrt{\pi}} \frac{\Gamma \left( \frac2d + \frac32 \right)  }{ \Gamma \left( \frac2d + 2  \right)}.
\end{align*}

\subsection{Another proof of \eqref{eq4.12v107}}
In this subsection, we give another proof of \eqref{eq4.12v107} in Theorem \ref{th6.2} by the argument in \cite{MN}. By direct calculation, we have
\begin{align*}
 \frac1{2\pi} \int_0^{2\pi}  f\left(w + e^{i \theta} \bar{w} \right) \,\mathrm{d} \theta = 2^{1 + \frac4d} C_d f(w).
\end{align*}
Then, we have
\[	e_{3,n} (t,x) = \mu 2^{-\frac4d - 1} \lambda_n^{-\frac{d}2-2}(\langle \lambda_n^{-1} \nabla \rangle^{-1} \mathcal{E}_{3,n}) \left( \frac{t}{\lambda_n^2}, \frac{x}{\lambda_n} \right),\]
where
\begin{align*}
\mathcal{E}_{3,n} (\tau, y)	&{}=  e^{-i\lambda_n^2 \tau} \left( f \left( w_n(\tau,y) + e^{2i\lambda_n^2 \tau}\overline{w_n(\tau,y)} \right)
	- \frac1{2\pi} \int_0^{2\pi} f \left( w_n(\tau,y) + e^{i \theta }\overline{w_n(\tau,y)} \right) d\theta \right).
\end{align*}
By changing of variables and by the  $L^2$-unitary property of $e^{-it \langle \lambda_n^{-1} \nabla \rangle}$, we have
\begin{align*}
  & \left\|\int_0^t e^{-i(t-s) \langle \nabla \rangle} e_{3,n}(s) \,\mathrm{d}s \right\|_{L_t^\infty H_x^\frac12 \left( \left[-\lambda_n^2 T, \lambda_n^2 T \right] \times \mathbb{R}^d \right)}\\
 =  &   2^{-\frac4d - 1}  \Bigg\| \langle \lambda_n^{-1} \nabla \rangle^{-\frac12 }  \int_0^t e^{i\lambda_n^2 \tau \langle \lambda_n^{-1} \nabla \rangle }
\mathcal{E}_{3,n}(\tau) \,\mathrm{d}\tau \Bigg\|_{L_t^\infty L_x^2( [-  T,  T] \times \mathbb{R}^d )} .
\end{align*}
A computation gives us
\begin{align*}
\mathcal{E}_{3,n} (\tau)&{}=  e^{-i\lambda_n^2 \tau}  \int_0^{1} f \left( w_n(\tau) + e^{2\pi i \frac{\lambda_n^2 }{\pi} \tau }\overline{w_n(\tau)} \right)
-  f \left( w_n(\tau) + e^{2\pi i \left(\theta +\frac{\lambda_n^2}{\pi} \tau\right)}\overline{w_n(\tau)} \right) d\theta \\
&{}=-  e^{-i\lambda_n^2 \tau} \int_0^{1} \int_0^\theta \partial_\eta  \left( f \left( w_n(\tau) + e^{2\pi i \left(\eta +\frac{\lambda_n^2}{\pi} \tau \right)}\overline{w_n(\tau)} \right)\right)  d\eta d\theta \\
&{}=-  e^{-i\lambda_n^2 \tau} \int_0^{1} (1-\eta ) \partial_\eta  \left( f \left( w_n(\tau) + e^{2\pi i \left(\eta +\frac{\lambda_n^2}{\pi} \tau\right)}\overline{w_n(\tau)} \right) \right) d\eta .
\end{align*}
Combining the above identities, we have
\begin{align*}
& \left\|\int_0^t e^{-i(t-s) \langle \nabla \rangle} e_{3,n}(s) \,\mathrm{d}s \right\|_{L_t^\infty H_x^\frac12 \left( \left[-\lambda_n^2 T, \lambda_n^2 T \right] \times \mathbb{R}^d \right)}\\
 = &   {2^{-\frac4d - 1}} \bigg\| \int_{\mathbb{R}} \int_0^1 ( 1- \eta)\langle \lambda_n^{-1} \nabla \rangle^{-\frac12}  \partial_\eta \left( g \left( \tau, \frac{\lambda_n^2}\pi \tau + \eta \right)\right) \,\mathrm{d}\eta \mathrm{d}\tau \bigg\|_{L_t^\infty L_x^2( [-  T,  T] \times \mathbb{R}^d )},
\end{align*}
where
\begin{align*}
g(\tau, \theta) = \chi_{[0,t]}(\tau) e^{i \lambda_n^2 \tau \left( \langle \lambda_n^{-1} \nabla \rangle -1 \right)}  f\left(w_n(\tau) + e^{2\pi i \theta} \overline{w_n(\tau)}\right).
\end{align*}
We now use the following identity
\[	\partial_\eta \left( g \left( \tau, \frac{\lambda_n^2}\pi \tau + \eta\right)\right)
	= \frac{\pi}{\lambda_n^2} \partial_\tau \left( g \left( \tau, \frac{\lambda_n^2}\pi \tau + \eta\right)\right)
	- \frac{\pi}{\lambda_n^2} (\partial_\tau g) \left( \tau, \frac{\lambda_n^2}\pi \tau + \eta\right)\]
to get the estimate
\begin{equation}\label{eq:e3n1}
\begin{aligned}
& \left\|\int_0^t e^{-i(t-s) \langle \nabla \rangle} e_{3,n}(s) \,\mathrm{d}s \right\|_{L_t^\infty H_x^\frac12 \left( \left[-\lambda_n^2 T, \lambda_n^2 T \right] \times \mathbb{R}^d  \right)}\\
 &\lesssim \lambda_n^{-2} \bigg\| \int_{\mathbb{R}} \int_0^1 ( 1- \eta)\langle \lambda_n^{-1} \nabla \rangle^{-\frac12} g_\tau \left( \tau, \frac{\lambda_n^2}\pi \tau + \eta \right) \,\mathrm{d}\eta \mathrm{d}\tau \bigg\|_{L_t^\infty L_x^2( [-  T,  T] \times \mathbb{R}^d )}\\
&\lesssim \lambda_n^{-2} \sup_{\theta \in \mathbb{R}} \bigg\| \int_{\mathbb{R}} \left| g_\tau ( \tau, \theta ) \right| \, \mathrm{d}\tau \bigg\|_{L_t^\infty L_x^2( [-  T,  T] \times \mathbb{R}^d )},
\end{aligned}
\end{equation}
where we have used the Minkowski inequality and the uniform boundedness of $\langle \lambda_n^{-1}\nabla \rangle^{- \frac12}$ in $L^2$
to obtain the last line. By direct computation, we have
\begin{equation}\label{eq:gtau}
\begin{aligned}
g_\tau(\tau, \theta)  = &\left( \delta(\tau) - \delta(\tau -t)\right) e^{i \lambda_n^2 \tau \left( \langle \lambda_n^{-1} \nabla \rangle - 1 \right)} f\left(w_n(\tau) + e^{2\pi i \theta} \overline{w_n(\tau)}\right) \\
&\  + i\lambda_n^2 \chi_{[0,t]}(\tau)   \left( \langle \lambda_n^{-1} \nabla \rangle - 1\right)  e^{i\lambda_n^2 \tau(
 \langle \lambda_n^{-1} \nabla \rangle - 1)} f\left(w_n(\tau) + e^{2\pi i \theta} \overline{w_n(\tau)}\right) \\
& \ + \chi_{[0,t]}(\tau) e^{i\lambda_n^2  \tau ( \langle \lambda_n^{-1} \nabla \rangle - 1)}
 \left( (\partial_z f)\left( w_n(\tau) + e^{2\pi i \theta} \overline{w_n(\tau)}\right) \left(\partial_\tau {w}_n(\tau) + e^{2\pi i \theta} \overline{ \partial_\tau {w}_n(\tau)}\right)\right)\\
& \ + \chi_{[0,t]}(\tau) e^{i\lambda_n^2  \tau  \left( \langle \lambda_n^{-1} \nabla \rangle - 1 \right)}
 \left( (\partial_{\overline{z}} f)\left( w_n(\tau) + e^{2\pi i \theta} \overline{w_n(\tau)}\right) \left(\overline{\partial_\tau {w}_n(\tau) }+ e^{-2\pi i \theta}  \partial_\tau {w}_n(\tau)\right)\right).
\end{aligned}
\end{equation}
Thus, by H\"older, the estimate $ \left|\lambda_n^{2} \left(\langle\lambda_n^{-1}\xi\rangle-1 \right) \right|\le \lambda_n |\xi|$, the fact that $w_n$ stands for $P_{\le \lambda_n^{2\theta}} w_n$ which satisfies \eqref{eq6.4} and \eqref{eq5.4v185} for all $s\ge0$ with the doubled $\theta$, and Sobolev, we finally obtain
\begin{align*}
& \lambda_n^{-2} \sup_{\theta \in \mathbb{R}} \bigg\| \int_{\mathbb{R}} \left| g_\tau ( \tau, \theta )  \right| \, \mathrm{d}\tau \bigg\|_{L_t^\infty L_x^2( [-  T,  T] \times \mathbb{R}^d )} \\
& \lesssim \lambda_n^{-2} \left\| f\left(w_n(0,x) + e^{2\pi i \theta} \overline{w_n(0,x)}\right) \right\|_{L_\theta^\infty L_x^2}
+ \lambda_n^{-2} \left\| f\left(w_n(t) + e^{2\pi i \theta} \overline{w_n(t)}\right) \right\|_{L_{\theta, t}^\infty L_x^2}\\
& \quad + T \lambda_n^{-1} \left\| \nabla \left( f\left(w_n(\tau) + e^{2\pi i \theta} \overline{w_n(\tau)}\right) \right)\right\|_{L_{\theta, \tau}^\infty L_x^2}\\
& \quad + T \lambda_n^{-2} \left\|  \left| w_n(\tau) + e^{2\pi i \theta} \overline{w_n(\tau)}\right|^{\frac4d} \left| \partial_\tau {w}_n(\tau) + e^{2\pi i \theta} \overline{\partial_\tau {w}_n(\tau)}\right| \right\|_{L_{\tau, \theta}^\infty L_x^2}\\
& \lesssim \lambda_n^{-2} \|w_n\|_{L_t^\infty H_x^{\frac{2d}{d+4} }}^{1 + \frac4d}
  + \lambda_n^{-1} \|w_n\|_{L_\tau^\infty H_x^{\frac{2d}{d+4} }}^\frac4d \left\|\nabla w_n(\tau) \right\|_{L_\tau^\infty H_x^{\frac{2d}{d+4} }}
+ \lambda_n^{-2}\|w_n \|_{L_{\tau}^\infty H_x^{\frac{2d}{d+4}} }^\frac4d \|\partial_\tau w_n\|_{L_\tau^\infty H_x^{\frac{2d}{d+4}} } \\
& \lesssim \lambda_n^{-1 + 8\theta}.
\end{align*}
Arguing as in \eqref{eq:e3n1}, we also have
\begin{align*}
  &  \left\|\int_0^t e^{-i(t-s) \langle \nabla \rangle} e_{3,n}(s) \,\mathrm{d}s \right\|_{L_{t,x}^\frac{2(d+2)}d \left( \left[- \lambda_n^2 T, \lambda_n^2 T \right] \times \mathbb{R}^d \right)}\\
\lesssim  & \lambda_n^{-2} \bigg\| \langle \lambda_n^{-1} \nabla \rangle^{-1}   \int_{\mathbb{R}} \int_0^1 ( 1- \theta) e^{-i \lambda_n^2 t \langle \lambda_n^{-1} \nabla \rangle} g_\tau\left( \tau, \frac{\lambda_n^2}\pi \tau + \theta\right) \,\mathrm{d}\theta \mathrm{d}\tau
\bigg\|_{L_{t,x}^\frac{2(d+2)}d([-T,T] \times \mathbb{R}^d)}.
\end{align*}
We now estimate each term contributed by $g_\tau$.
By the Strichartz estimate and Sobolev embedding, the contribution from the first line of the right hand side of \eqref{eq:gtau} is bounded by
\begin{align*}
& \lambda_n^{-2}\left\|f\left(w_n(0,x) + e^{2 \pi i \theta } \overline{w_n(0,x)}  \right)  \right\|_{L_{\theta}^\infty L_x^2}
+ \lambda_n^{-2}\left\| f\left(w_n(t,x) + e^{2\pi i \left( \theta + \frac{\lambda_n^2}\pi t\right)} \overline{w_n(t)}  \right) \right\|_{L_\theta^\infty L_{t,x}^\frac{2(d+2)}d([-T,T])} \\
& \lesssim \lambda_n^{-2} \|w_n(0) \|_{H_x^{\frac{2d}{d+4 }}}^{ \frac4d + 1} + \lambda_n^{-2}T^{\frac{d}{2(d+2)}} \| w_n\|_{L_{t}^\infty H_x^{\frac{d(3d+4)}{(d+2)(d+4)}}}^{\frac4d + 1}
\lesssim_T \lambda_n^{-2+4\theta} + \lambda_n^{-2+\frac{2(3d+4)}{d+2}\theta} \lesssim \lambda_n^{-2+6\theta}.
\end{align*}
We now turn to the contribution from the other part of $g_\tau$. Remark that one can apply inhomogeneous Strichartz estimate and then the estimate becomes essentially same as the previous case: it is bounded by
\begin{align*}
& T \lambda_n^{-1} \left\| \nabla \left( f\left(w_n(\tau) + e^{2\pi i \theta} \overline{w_n(\tau)}\right) \right)\right\|_{L_{\theta, \tau}^\infty L_x^2}\\
& \quad + T \lambda_n^{-2} \left\|  \left| w_n(\tau) + e^{2\pi i \theta} \overline{w_n(\tau)}\right|^{\frac4d} \left| \partial_\tau {w}_n(\tau) + e^{2\pi i \theta} \overline{\partial_\tau {w}_n(\tau)}\right| \right\|_{L_{\tau, \theta}^\infty L_x^2}\\
& \lesssim_T \lambda_n^{-1} \|w_n\|_{L_\tau^\infty H_x^{\frac{2d}{d+4} }}^\frac4d \left\|\nabla w_n(\tau) \right\|_{L_\tau^\infty H_x^{\frac{2d}{d+4} }}
+ \lambda_n^{-2}\|w_n \|_{L_{\tau}^\infty H_x^{\frac{2d}{d+4}} }^\frac4d \|\partial_\tau w_n\|_{L_\tau^\infty H_x^{\frac{2d}{d+4}} }
\lesssim \lambda_n^{-1 + 8\theta}.
\end{align*}
Thus, we have finished the proof of \eqref{eq4.12v107}.

\noindent \textbf{Acknowledgments.}
The authors are grateful to Professor Nakanishi for helpful discussion. Xing Cheng wish to thank Professor Stefanov for the interest he has taken in this work and for fruitful discussions.

\end{document}